\documentclass[twoside, 11pt]{article}
\usepackage[top=.67in, bottom=.71in, left=.89in, right=.83in]{geometry}
\usepackage{enumerate,float}
\usepackage{section, amsthm, textcase, setspace, amssymb, lineno, amsmath, amssymb, amsfonts, latexsym, fancyhdr, longtable, ulem}
\usepackage{section, amsthm, textcase, setspace, amssymb, lineno, amsmath, amssymb, amsfonts, latexsym, fancyhdr, longtable, ulem}
\usepackage{tikz}
\usetikzlibrary{decorations.markings,decorations.pathreplacing,patterns}
\tikzstyle{vertex}=[circle, draw, inner sep=0pt, minimum size=11pt]
\newcommand{\vertex}{\node[vertex]}

\newtheorem{thm}{Theorem}[section]
\newtheorem{lemma}[thm]{Lemma}
\newtheorem{cor}[thm]{Corollary}
\newtheorem{remark}[thm]{Remark}

\theoremstyle{definition}
\newtheorem{definition}[thm]{Definition}
\newtheorem{example}[thm]{Example}

\newcommand{\mf}{\mathfrak}
\newcommand{\wh}{\widehat}
\newcommand{\ul}{\underline}
\newcommand{\Cn}{C_{\leq n}}
\newcommand{\dd}{\hspace{.1cm}|\hspace{.1cm}}
\newcommand{\ind}{{\rm ind \hspace{.1cm}}}

\newcommand{\A}{{\rm A}}
 
\newcommand{\C}{{\rm C}}
\newcommand{\D}{{\rm D}}
\newcommand{\I}{{\rm I}}
\newcommand{\II}{{\rm II}}
\newcommand{\III}{{\rm III}}

\newcommand{\lf}{\left\lfloor}
\newcommand{\rf}{\right\rfloor}
\newcommand{\Z}{\mathbb Z}

\usepackage{tkz-graph}

\begin{document}
%\linenumbers

\title{\bf Combinatorial index formulas for Lie algebras of seaweed type}

\author{Alex Cameron$^*$, Vincent E. Coll, Jr.$^*$, and Matthew Hyatt$^{**}$}

\maketitle

%%%%%%%%%%%%%%%%%%%%%%%%%%%%

\noindent
\textit{$^*$Department of Mathematics, Lehigh University, Bethlehem, PA, USA}\\
\textit{$^{**}$FactSet Research Systems, New York, NY, USA}\\

\begin{abstract}
\noindent
Analogous to the types A, B, and C cases, we address the computation of the index of seaweed subalgebras in the type-D case. Formulas for the algebra's index can be computed by counting the connected components of its associated meander. 
We focus on a set of distinguished vertices of the meander, called the tail of the meander, and using the tail, we provide comprehensive combinatorial formulas for the index of a seaweed in all the classical types.  Using these formulas, we provide all general closed-form index formulas where the index is given by a polynomial greatest common divisor formula in the sizes of the parts that define the seaweed.  

\end{abstract}



%%%%%%%%%%%%%%%%%%%%%%%%%%%%%%%%%%%%%%%
\section{Introduction}

The \textit{index} of a Lie algebra $\mf{g}$ is an important algebraic invariant introduced by Dixmier (\textbf{\cite{JDix}}, 1974) and is defined by 

$$
\ind \mf{g}=\min_{f\in \mf{g^*}} \dim  (\ker (B_f)),
$$
where $f$ is an element of the linear dual $\mathfrak{g}^*$, and $B_f$ is the associated skew-symmetric \textit{Kirillov form} defined by 

$$
B_f(x,y)=f([x,y]) \textit{ for}~ x,y\in\mf{g}, 
$$
and

$$
\ker B_f= \{x\in \mf{g} ~|~ f[x,y]=0 \textit{ for all}~  y\in \mf{g}                          \}.
 $$ 
Here, we focus on combinatorial mechanisms to compute the index of certain subalgebras of the classical Lie algebras, which are the evocatively-named \textit{seaweed algebras}.
These algebras, along with their suggestive name, were first introduced by Dergachev and A. Kirillov in \textbf{\cite{dk}}, where they defined such algebras as subalgebras of $\mathfrak{gl}(n)$ preserving certain flags of subspaces developed from two compositions of $n$. The evocative  ``seaweed'' comes from the wavy shape the algebra demonstrates when exhibited in its standard matrix representation.  The type-A case, $A_{n-1}=\mathfrak{sl}(n)$, is considered by requiring the elements of the seaweed in $\mathfrak{gl}(n)$ to have trace zero.  
Subsequently in \textbf{\cite{Panyushev1}}, Panyushev extended the Lie-theoretic definition of seaweed algebras to the reductive case.  If $\mf{p}$ and $\mf{p'}$ are parabolic subalgebras of a reductive Lie algebra $\mf{g}$ such that $\mf{p}+\mf{p'}=\mf{g}$, then $\mf{p}\cap\mf{p'}$ is called a \textit{seaweed subalgebra o}f $\mf{g}$ or simply $seaweed$ when $\mathfrak{g}$ is understood.  As a result of this definition, Joseph has elsewhere \textbf{\cite{Jo1}} called seaweed algebras \textit{biparabolic}.  Joseph also showed, in response to a conjecture by Tauvel and Yu in \textbf{\cite{TY1}}, that the index of a seaweed is bounded by the algebra's rank \textbf{\cite{Jo1}}. 

%We will find that for certain seaweed algebras, one may attach a planer graph, called a \emph{meander}, and the computation of the index of the seaweed algebra, quite remarkably, reduces to counting the number and type of the connected components of its associated meander graph.

To facilitate the computation of the index of seaweed subalgebras of $\mathfrak{gl}(n)$, 
the authors in \textbf{\cite{dk}} introduced the notion of a \textit{meander} --  a planar graph representation of the seaweed algebra.  The main result of \textbf{\cite{dk}} is that the index of a seaweed can be computed based on the number and type of the connected components of its meander.  
A slightly modified formula yields the index of a seaweed subalgebra of $\mf{sl}(n)$ (see \textbf{\cite{Coll1}}).  
In the maximal parabolic case in type A, but using different methods, Elashvili \textbf{\cite{Elash}} provided an explicit index formula which is presented in terms of a linear greatest common divisor of two arguments, each of which is a linear combination of the terms in the seaweed's defining compositions. In \textbf{\cite{Coll1, Coll2}}, Coll et al. developed a similar index formula in the next most complicated case -- a total of four terms in the defining compositions -- and conjectured that no single linear greatest common divisor formula could deliver the index of a general seaweed with more than four total terms in its defining compositions.  In \textbf{\cite{Kar}}, 
Karnauhova and Liebscher proved this conjecture by establishing the following beautiful general theorem.

To set the notation, let $\ul{a} = (a_1, \dots , a_m)$ and 
$\ul{b} = (b_1, \dots , b_l)$ be two compositions of $n$, and let 
$M^\A_{n} (\ul{a} \dd \ul{b})$ denote the meander associated with the type-A seaweed 
$\mf{p}^\A_{n} (\ul{a} \dd \ul{b})$.

\begin{thm}[Karnauhova and Liebscher \textbf{\cite{Kar}}, 2015]\label{5 parts} If $m\geq 4$, then there do not exist homogeneous polynomials $f_1,f_2\in \Z[x_1,\dots ,x_m]$ of arbitrary degree such that the number of connected components of 
$M^\A_{n} (\ul{a}, n)$ is given by 

$$\gcd(f_1(a_1,\dots ,a_m),f_2(a_1,\dots ,a_m)).$$
\end{thm}

Extending this line of inquiry in \textbf{\cite{Coll4}}, Coll et al. obtained similar combinatorial index formulas 
in the type-C case, $C_n = \mf{sp}(2n)$. The requisite type-C meander development was undertaken by Coll, Hyatt, and Magnant in \textbf{\cite{Coll4}}, where such meanders were called \textit{symplectic meanders}.
Somewhat later, these type-C meanders were developed independently by Panyushev and Yakimova (see \textbf{\cite{Panyushev2}}).  However, the approach taken by Coll et al. is distinguished by an emphasis on closed-form linear index formulas and an analysis of a special set of vertices in a type-C meander which they called the \textit{tail} of the meander. The analogue of the above theorem of Karnauhova and Liebscher in the type-C case is implicit in \textbf{\cite{Coll4}}.  
Moreover, it follows from Joseph (\textbf{\cite{Jo1}}, Theorem 8.4) that the meandric index analysis in type-C carries over \textit{mutatis mutandis} to the type-B case.

In this paper, we consider the type-D case, $D_n = \mf{so}(2n)$.  We follow the program outlined above and develop the following:

\begin{enumerate}
\item  Type-D meanders.  
As with type C, our approach once again parallels the work of Panyushev and Yakimova in  \textbf{\cite{Panyushev3}}, but is distinguished, as before, by our goals and methods.  We find, in particular, that the tail of the meander associated with a type-D seaweed has a more subtle structure.  By exploiting the various configurations of the tail in this classical type, we develop a
combinatorial formula for the index of a type-D seaweed based on the number and type of connected components in the seaweed's associated meander.  This formula is a bit less complicated than that found in \textbf{\cite{Panyushev3}} and can be used to classify index zero (Frobenius\footnote{Frobenius algebras are of  special interest in deformation and quantum group theory stemming from their connection with the classical Yang-Baxter equation (see \textbf{\cite{G1}} and \textbf{\cite{G2}}).  More specifically, an index-realizing functional is called \textit{regular}, and a regular functional $F$ on a Frobenius Lie algebra $\mathfrak{g}$ is called a \textit{Frobenius functional}; equivalently, $B_F(-,-)$ is non-degenerate.  Suppose $B_F(-,-)$ is non-degenerate and let $[F]$ be the matrix of $B_F(-,-)$ relative to some basis 
$\{x_1,\dots,x_n  \}$ of $\mathfrak{g}$.  In \textbf{\cite{BelDrin}}, Belavin and Drinfeld showed that   
\[
\sum_{i,j}[F]^{-1}_{ij}x_i\wedge x_j
\]
\noindent
is the infinitesimal of a \textit{Universal Deformation Formula} (UDF) based on $\mathfrak{g}$.  A UDF based on $\mathfrak{g}$ can be used to deform the universal enveloping algebra of $\mathfrak{g}$ and also the function space on any Lie group which contains $\mathfrak{g}$ in its Lie algebra of derivations.}) type-D seaweeds up to a similarity transformation;
\item  Obtain an exhaustive list of closed- form index formulas in all ``reasonable cases".   Interestingly, seaweeds in type D are not necessarily seaweed ``shaped'' in their natural matrix representations. We discern why and and show that such algebras have the same index as a certain seaweed algebra of the same dimension that does have seaweed shape.  
\item Establish the analogue in type D of the above theorem of Karnauhova and Liebscher (see Theorem \ref{KarLiebTypeD}).
\end{enumerate}

The first four sections of the paper recount, and expand upon, the meander-based formulas in the first three classical families.  We include these abridged results since we require them in their entirety to deal with the subtleties encountered in the type-D case.

The structure of the paper is as follows.  In the single paragraph which comprises Section 2, we provide the formal definition of a seaweed algebra.  Section 3 consists of brief summary of the results in type A, while Section 4 summarizes the type-C and type-B index results of Coll et  al. (see \textbf{\cite{Coll4}}.) Section 5 contains the main results of the paper, where the type-D case is analyzed.  

%%%%%%%%%%%%%%%%%%%%%%%%%%%%%
\section{Seaweeds}
%%%%%%%%%%%%%%%%%%%%%%%%%%%%%
We assume that a seaweed $\mf{g}$ is equipped with a triangular decomposition 

$$ \mf{g}=\mf{u_+}\oplus\mf{h}\oplus\mf{u_-},$$
\noindent
where $\mf{h}$ is a Cartan subalgebra of $\mf{g}$, and $\mf{u_+}$ and $\mf{u_-}$ are the subalgebras consisting of the upper and lower triangular matrices, respectively. Let $\Pi$ be the set of $\mf{g}$'s simple roots, and for $\alpha\in\Pi$,
let $\mf{g}_{\alpha}$ denote the root space corresponding to $\alpha$. A seaweed subalgebra $\mf{p}\cap\mf{p'}$ is called \textit{standard} if
$\mf{p}\supseteq \mf{h}\oplus\mf{u}_+$ and $\mf{p'}\supseteq \mf{h}\oplus\mf{u_-}$.
In the case that $\mf{p}\cap\mf{p'}$ is standard, let
$\Psi=\{\alpha\in\Pi :\mf{g}_{-\alpha}\notin \mf{p}\}$ and 
$\Psi'=\{\alpha\in\Pi :\mf{g}_{\alpha}\notin \mf{p'}\}$,
and denote the seaweed by $\mf{p}(\Psi \dd \Psi')$.
Any seaweed is conjugate, over its algebraic group, to a standard one, so it suffices to work
with standard seaweeds only. Note that an arbitrary seaweed may be 
conjugate to more than one standard seaweed (see \textbf{\cite{Panyushev1}}, page 226).  

\section{Type A - $\mathfrak{sl}(n)$}

\subsection{Type-A seaweeds}

Let $\mf{sl}(n)$ be the algebra of $n\times n$ matrices with trace zero and consider the triangular decomposition of $\mf{sl}(n)$ as above.  Let $\Pi=\{\alpha_1,\dots ,\alpha_{n-1}\}$ be the set of simple roots of $\mf{sl}(n)$ with the standard ordering, and let $\mf{p}_n^\A(\Psi \dd \Psi')$ denote a seaweed subalgebra of $\mf{sl}(n)$, where $\Psi$ and $\Psi'$ are subsets of $\Pi$.

Let $C_n$ denote the set of strings of positive integers whose sum is $n$.
It will be convenient to index seaweeds of $\mf{sl}(n)$ by pairs of elements 
of $C_n$. Let $\mathcal{P}(X)$ denote the power set of a set $X$.
Let $\varphi_\A$ be the usual bijection from $C_n$ to a set of cardinality $n-1$.
That is, given $\ul{a}=(a_1,a_2,\dots ,a_m)\in C_n$, define
$\varphi_\A :C_n\rightarrow \mathcal{P}(\Pi)$ by

$$\varphi_\A(\ul{a})=\{\alpha_{a_1},\alpha_{a_1+a_2},\dots 
,\alpha_{a_1+a_2+\dots +a_{m-1}}\}.$$
\noindent
Then define 

$$\mf{p}_n^\A(\ul{a} \dd \ul{b})
=\mf{p}_n^\A(\varphi_\A(\ul{a}) \dd \varphi_\A(\ul{b})).$$
%Now, given $S=\{\alpha_{s_1},\alpha_{s_2},\dots ,\alpha_{s_j}\}\subseteq \mathcal{P}(\Pi)$
%with $s_1<s_2<\cdots <s_j$, we have
%\[\varphi_{\A}^{-1}(S)=(s_1,s_2-s_1,s_3-s_2,\dots ,s_j-s_{j-1},n-s_j).\]
\noindent
By construction, the sequence of numbers in $\ul{a}$ determines the heights of triangles
below the main diagonal in $\mf{p}_n^\A(\ul{a} \dd \ul{b})$ which may have nonzero entries,
and the sequence of numbers in $\ul{b}$ determines the heights of triangles
above the main diagonal. 
For example, the seaweed $\mf{p}_7^\A((4,3) \dd (2,2,2,1))
=\mf{p}_7^\A(\{\alpha_4\} \dd \{\alpha_2,\alpha_4,\alpha_6\})$
has the following shape, where * indicates a possible nonzero entry. See 
the left-hand side of Figure \ref{Aseaweed}.
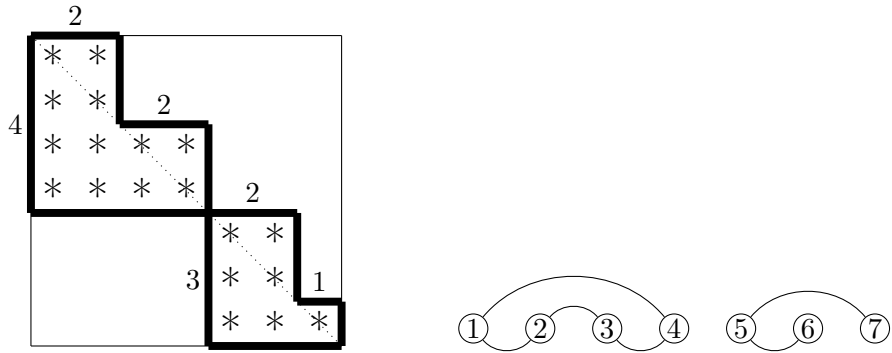
\begin{figure}[H]
\[\begin{tikzpicture}[scale=0.59]
\draw (0,0) -- (0,7);
\draw (0,7) -- (7,7);
\draw (7,7) -- (7,0);
\draw (7,0) -- (0,0);
\draw [line width=3](0,7) -- (0,3);
\draw [line width=3](0,3) -- (4,3);
\draw [line width=3](4,3) -- (4,0);
\draw [line width=3](4,0) -- (7,0);

\draw [line width=3](0,7) -- (2,7);
\draw [line width=3](2,7) -- (2,5);
\draw [line width=3](2,5) -- (4,5);
\draw [line width=3](4,5) -- (4,3);
\draw [line width=3](4,3) -- (6,3);
\draw [line width=3](6,3) -- (6,1);
\draw [line width=3](6,1) -- (7,1);
\draw [line width=3](7,1) -- (7,0);

\draw [dotted] (0,7) -- (7,0);

\node at (.5,6.4) {{\LARGE *}};
\node at (1.5,6.4) {{\LARGE *}};
\node at (.5,5.4) {{\LARGE *}};
\node at (1.5,5.4) {{\LARGE *}};
\node at (.5,4.4) {{\LARGE *}};
\node at (1.5,4.4) {{\LARGE *}};
\node at (2.5,4.4) {{\LARGE *}};
\node at (3.5,4.4) {{\LARGE *}};
\node at (.5,3.4) {{\LARGE *}};
\node at (1.5,3.4) {{\LARGE *}};
\node at (2.5,3.4) {{\LARGE *}};
\node at (3.5,3.4) {{\LARGE *}};
\node at (4.5,2.4) {{\LARGE *}};
\node at (4.5,1.4) {{\LARGE *}};
\node at (5.5,1.4) {{\LARGE *}};
\node at (5.5,2.4) {{\LARGE *}};
\node at (4.5,0.4) {{\LARGE *}};
\node at (5.5,0.4) {{\LARGE *}};
\node at (6.5,0.4) {{\LARGE *}};

\node at (-.35,5) {4};
\node at (3.65,1.5) {3};
\node at (1,7.45) {2};
\node at (3,5.45) {2};
\node at (5,3.45) {2};
\node at (6.5,1.45) {1};

\end{tikzpicture}
\hspace{1.5cm}
\begin{tikzpicture}[scale=.89]
\vertex (1) at (1,0) {1};
\vertex (2) at (2,0) {2};
\vertex (3) at (3,0) {3};
\vertex (4) at (4,0) {4};
\vertex (5) at (5,0) {5};
\vertex (6) at (6,0) {6};
\vertex (7) at (7,0) {7};

\path
(1) edge[bend left=50] (4)
(2) edge[bend left=50] (3)
(5) edge[bend left=50] (7)
(1) edge[bend right=50] (2)
(3) edge[bend right=50] (4)
(5) edge[bend right=50] (6)
;\end{tikzpicture}
\]
\caption{
$\mf{p}_7^\A((4,3) \dd (2,2,2,1))$ and its associated
meander}
\label{Aseaweed}
\end{figure}

\begin{remark}\label{ASeaweedShape}
The seaweed in Figure \ref{Aseaweed} has \rm{seaweed shape}:
\textit{Let $D_{ \ul{a} }$ be the subalgebra of block-diagonal matrices whose blocks have sizes
$a_1\times a_1,\dots,a_m\times a_m$ and similarly for $D_{\ul{b}}$.  A seaweed in type A has seaweed shape if it is the subalgebra of $\mf{gl}(n)$ spanned by the intersection of $D_{\ul{a}}$ with the lower triangular matrices, the intersection of $D_{\ul{b}}$ with the upper triangular matrices, and all diagonal matrices. }
\end{remark}

%%%%%%%%%%%%%%%%%%%%%%%%%%%%%
\subsection{Type-A meanders}
%%%%%%%%%%%%%%%%%%%%%%%%%%%%%
Given a seaweed $\mf{p}_n^\A(\ul{a} \dd \ul{b})$ in $\mf{sl}(n)$,  Dergachev and A. Kirillov \textbf{\cite{dk}} showed
how to associate a planar graph
called a \textit{meander}, denoted $M_n^\A(\ul{a} \dd \ul{b})$.
We label the vertices of $M_n^\A(\ul{a} \dd \ul{b})$ as $1, 2, \dots ,n$ from left to right, and place edges above them, called \textit{top edges}, according to $\ul{a}$ as follows. 
Let $\ul{a}=(a_1,a_2,\dots , a_m)$, and let $V_i$ be the $i^{\text{th}}$ \textit{block} of vertices, that is the subset
of vertices whose label is greater than $a_1+a_2+\dots +a_{i-1}$ and less than $a_1+a_2+\dots +a_i+1$.
For each block $V_i$, place top edges connecting vertex $j$ to vertex $k$ if 
$j+k=2(a_1+a_2+\dots+a_{i-1})+a_i+1$. In the same way,
place \textit{bottom edges} according to $\ul{b}$. See the right-hand side of Figure 1.

Since each vertex is incident with at most one top edge, and at most one
bottom edge, we define a \textit{top bijection} $t$ on $\{1, ..., n\}$ by $t(j)=k$ if there is a top edge from vertex $j$ to 
vertex $k$, and $t(j)=j$ if vertex $j$ is not incident with a top edge. Similarly, we define a \textit{bottom bijection} $b$ on $\{1, ..., n\}$.  Given a meander $M_n^\A(\ul{a} \dd \ul{b})$,
let $\sigma_{\ul{a},\ul{b}}$ be its associated permutation defined by $\sigma_{\ul{a},\ul{b}}(j)=t(b(j))$.
For example if $\ul{a}=(4,3)$ and $\ul{b}=(2,2,2,1)$, then the associated permutation written as a product of 
disjoint cycles is $\sigma_{\ul{a},\ul{b}}=(1,3)(2,4)(6,7,5)$.

The following result follows immediately from Theorem 5.1 of \textbf{\cite{dk}}.  (Note that the formula in the theorem below differs by
one from the theorem of Dergachev and Kirillov --  since we are working in $\mf{sl}(n)$ instead of $\mf{gl}(n)$.)

\begin{thm}\label{DKformula}
The index of $\mf{p}_n^\A(\ul{a} \dd \ul{b})$ with associated meander $M_n^\A(\ul{a} \dd \ul{b})$ is equal to $2C + P -1$, where $C$ is the number of cycles in $M_n^\A(\ul{a} \dd \ul{b})$ and $P$ is the number of paths in $M_n^\A(\ul{a} \dd \ul{b})$.  
 
\end{thm}

An immediate consequence of Theorem \ref{DKformula} is that $\mf{p}_n^\A(\ul{a} \dd \ul{b})$ is Frobenius if and only if $M_n^\A(\ul{a} \dd \ul{b})$ is a single path.  For these seaweeds, the permutation $\sigma_{\ul{a},\ul{b}}$ is a single cycle and hence defines a permutation of $\{1, ..., n\}$.  The seaweed $\mf{p}_7^\A((4,3) \dd (7))$ 
%in Figure 2 
has associated permutation $\sigma_{\ul{a},\ul{b}} = (4,1,5,2,6,3,7)$.

\subsection{Type-A index formulas}
While Theorem \ref{DKformula} provides an elegant formalism for computing the index of a seaweed, significant computational complexity persists.  What is needed is a mechanism for determining the index of a seaweed directly from its defining compositions. The first result of this kind is due to Elashvili, who used different notation to establish the following.  

\begin{thm}[Elashvili \textbf{\cite{Elash}}, 1990] 
The seaweed $\mf{p}_n^\A((a,b) \dd (n))$ has index $\rm{gcd}(a,b)-1.$  
\end{thm}

\noindent
In \textbf{\cite{Coll2}}, Coll et al. provide a recursive classification of meander graphs, showing that each meander is identified by a unique sequence of fundamental graph theoretic moves, each of which is uniquely determined by the structure of the meander at the time of move application.  The sequence of (winding-down) moves is called the \textit{signature} of the meander (see Appendix A, Lemma \ref{WindingDown}).  Although discovered independently, the signature may be regarded as a graph theoretic rendering of Panyushev's well-known reduction \textbf{\cite{Panyushev1}}.  Using the signature, Coll et al. established the following extension of Elashvili's theorem.
 
\begin{thm}[Coll et al. \textbf{\cite{Coll2}}, 2015]\label{A 4 parts}  The seaweeds $\mf{p}_n^\A((a,b,c) \dd (n))$
and $\mf{p}_n^\A((a,b) \dd (c,n-c))$ have index

$$\gcd(a+b,b+c) -1.$$
\end{thm}
 
\begin{remark} 
The winding-down moves can be reversed to yield ``winding-up" moves, which can be used to build any meander of any size and configuration \textup(see Appendix A, Lemma \ref{WindingUp}\textup).  
\end{remark}

One might conjecture the existence of similar ``closed-form" index formulas for more general seaweeds, but, using signature moves and complexity arguments, Karnauhova and Liebscher have shown that there are severe restrictions.

\begin{thm}[Karnauhova and Liebscher \textbf{\cite{Kar}}, 2015]\label{5 parts} If $m\geq 4$, then there do not exist homogeneous polynomials $f_1,f_2\in \Z[x_1,\dots ,x_m]$ of arbitrary degree such that the number of connected components of 
$M^A_{n} ((a_1,\dots,a_m) \dd (n))$ is given by 
$\gcd(f_1(a_1,\dots ,a_m),f_2(a_1,\dots ,a_m))$.
\end{thm}

\section{Type C - $\mathfrak{sp}(2n)$ and Type B - $\mf{so}(2n+1)$}

\subsection{Type-C seaweeds}

In this subsection, we introduce type-C seaweeds and, following Coll et al. in \textbf{\cite{Coll2}}, we develop type-C meanders (see also \textbf{\cite{Panyushev2}}).  As in type A, the index of a type-C seaweed can be computed from simple graph theoretic properties of the type-C meander. 

Let $\mf sp(2n)$ be the algebras of matrices with the following block form

\[\mf sp(2n)=\left\{\begin{bmatrix} A & B \\ C & -\wh{A}\end{bmatrix}: B=\wh{B}, C=\wh{C} \right\},\]
where $A, B,$ and $C$ are $n\times n$ matrices, and $\wh{A}$ is the transpose of 
$A$ with respect to the antidiagonal.
Choose the same triangular decomposition as was done in the $\mf{sl}(n)$ case,
that is $\mf{sp}(2n)=\mf{u_+}\oplus\mf{h}\oplus\mf{u_-}$.
%where $\mf{h}$ is the Cartan subalgebra of diagonal matrices,
%$\mf{u}$ is the set of upper triangular matrices,
%and $\mf{u_-}$ is the set of lower triangular matrices.
Let $\Pi=\{\alpha_1,\dots ,\alpha_{n}\}$ denote its set of simple roots,
where $\alpha_n$ is the exceptional root.
Let $\mf{p}_n^\C(\Psi \dd \Psi')$ denote a seaweed subalgebra of $\mf{sp}(2n)$, where
$\Psi$ and $\Psi'$ are subsets of $\Pi$.
%We choose this basis??? so that seaweed subalgebras are standard (see \cite{Panyushev} section 2). 

Let $C_{\leq n}$ denote the set of strings of positive integers whose sum 
is less than or equal to $n$, and call each integer in the string a \textit{part}.
We will index seaweeds in $\mf{sp}(2n)$ by pairs of elements from $\Cn$.
%Indeed, let $[n]$ denote the set $\{1,2,\dots ,n\}$
%($[n]=\{1,2,\dots ,n\}$ is standard notation in combinatorics, is it appropriate here ???)
Let $\mathcal{P}(X)$ denote the power set of a set $X$.
Given $\ul{a}=(a_1,a_2,\dots ,a_m)\in \Cn$, define a bijection $\varphi_\C:\Cn\rightarrow \mathcal{P}(\Pi)$ by
%\[\phi(\ul{a})=\begin{cases}
%[n-1]\setminus\{a_1,a_1+a_2,\dots ,\sum_{i=1}^{m-1}a_i\} & \text{ if } \sum_{i=1}^{m}a_i=n \\
%[n]\setminus\{a_1,a_1+a_2,\dots ,\sum_{i=1}^{m}a_i\} & \text{ if } \sum_{i=1}^{m}a_i<n.
%\end{cases}\]

\[\varphi_\C(\ul{a})=\{\alpha_{a_1},\alpha_{a_1+a_2},\dots ,\alpha_{a_1+a_2+\dots +a_m}\},\]
and define 

\[\mf{p}_n^\C(\ul{a} \dd \ul{b})=
\mf{p}_n^\C(\varphi_\C(\ul{a}) \dd \varphi_\C(\ul{b})).\]
%Note that we need to keep the subscript $n$, since this is not determined by either $\ul{a}$ or $\ul{b}$.

%Given $S=\{\alpha_{s_1},\alpha_{s_2},\dots ,\alpha_{s_j}\}\subseteq \mathcal{P}(\Pi)$
%with $s_1<s_2<\cdots <s_j$, we have
%\[\varphi_{\C}^{-1}(S)=(s_1,s_2-s_1,s_3-s_2,\dots ,s_j-s_{j-1})\]

\begin{example}
The seaweed $\mf{p}_3^\C(\{\alpha_3\} \dd \{\alpha_1\}) = \mf{p}_3^\C((3) \dd (1))$ is the algebra of matrices in
$\mf{sp}(6)$ of the form in Figure \ref{Cseaweed} below, 
where * indicates a possible nonzero entry.
\begin{figure}[H]
\[\begin{tikzpicture}[scale=.61]
\draw (0,0) -- (0,6);
\draw (0,6) -- (6,6);
\draw (6,6) -- (6,0);
\draw (6,0) -- (0,0);
\draw [line width=3](0,6) -- (0,3);
\draw [line width=3](0,3) -- (3,3);
\draw [line width=3](3,3) -- (3,0);
\draw [line width=3](3,0) -- (6,0);
\draw [line width=3](0,6) -- (1,6);
\draw [line width=3](1,6) -- (1,5);
\draw [line width=3](1,5) -- (5,5);
\draw [line width=3](5,5) -- (5,1);
\draw [line width=3](5,1) -- (6,1);
\draw [line width=3](6,1) -- (6,0);
\draw [dotted] (0,3) -- (6,3);
\draw [dotted] (3,6) -- (3,0);
\draw [dotted] (0,6) -- (6,0);
\draw [dotted] (0,0) -- (6,6);

\node at (.5,5.4) {{\LARGE *}};
\node at (.5,4.4) {{\LARGE *}};
\node at (.5,3.4) {{\LARGE *}};
\node at (1.5,3.4) {{\LARGE *}};
\node at (1.5,4.4) {{\LARGE *}};
\node at (2.5,4.4) {{\LARGE *}};
\node at (2.5,3.4) {{\LARGE *}};
\node at (3.5,4.4) {{\LARGE *}};
\node at (3.5,3.4) {{\LARGE *}};
\node at (3.5,2.4) {{\LARGE *}};
\node at (4.5,.4) {{\LARGE *}};
\node at (3.5,1.4) {{\LARGE *}};
\node at (4.5,4.4) {{\LARGE *}};
\node at (4.5,3.4) {{\LARGE *}};
\node at (4.5,2.4) {{\LARGE *}};
\node at (4.5,1.4) {{\LARGE *}};
\node at (3.5,.4) {{\LARGE *}};
\node at (5.5,0.4) {{\LARGE *}};
\node at (-.35,4.5) {3};
\node at (.5,6.35) {1};

\end{tikzpicture}\]
\caption{The shape of elements from $\mf{p}_3^\C((3) \dd (1))$}
\label{Cseaweed}
\end{figure}
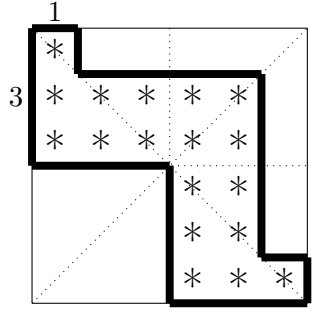
\end{example}

\begin{remark}\label{CSeaweedShape}
Similar to the type-A case \textup(see Remark \ref{ASeaweedShape}\textup), type-C seaweeds have seaweed shape:
\textit{Let $D_{ \ul{a} }$ be the subalgebra of block-diagonal matrices whose blocks have sizes
$a_1\times a_1,\dots,a_m\times a_m, 2(n-\sum a_i) \times 2(n-\sum a_i), a_m \times a_m, \dots, a_1 \times a_1$ and similarly for $D_{\ul{b}}$.  A type-C seaweed has seaweed shape if it is the subalgebra of $\mf{gl}(n)$ spanned by the intersection of $D_{\ul{a}}$ with the lower triangular matrices, the intersection of $D_{\ul{b}}$ with the upper triangular matrices, and all diagonal matrices. }
\end{remark}

%%%%%%%%%%%%%%%%%%%%%%%%%%%%%%%%%%%%
\subsection{Type-C meanders}
%%%%%%%%%%%%%%%%%%%%%%%%%%%%

Given a seaweed $\mf{p}_n^\C(\ul{a} \dd \ul{b})$, associate 
a type-C meander, which we denote $M_n^\C(\ul{a} \dd \ul{b})$.
The construction is the same as type-A meanders.
%We label the vertices of $M_n^\C(\ul{a} \dd \ul{b})$ as $1, 2, \dots ,n$ from left to right along a
%horizontal line. We begin by placing edges above the horizontal line, called top edges, 
%according to $\ul{a}$ as follows. 
%Let $\ul{a}=(a_1,a_2,\dots , a_m)$, and let $V_i$ be the $i^{\text{th}}$ block of vertices, that is the subset
%of vertices whose label is greater than $a_1+a_2+\dots +a_{i-1}$ and less than $a_1+a_2+\dots +a_i+1$.
%For each block $V_i$, place an edge from vertex $j$ to vertex $k$
%if $j+k=2(a_1+a_2+\dots+a_{i-1})+a_i+1$. Next, in the same way,
%place bottom edges according to $\ul{b}$. 
But for a type-C meander, we designate a special subset of vertices $T=T_n(\ul{a}\dd\ul{b})$ 
called the\textit{ tail} of the meander as follows:  
if $\ul{a}\in \Cn$, let $r=\sum a_i$, and define a subset of vertices
$T_n(\ul{a})=\{v_{r+1},v_{r+2},\dots ,v_n\}$.  Then $T_n(\ul{a}\dd\ul{b})$ is the symmetric difference of
$T_n(\ul{a})$ and $T_n(\ul{b})$, i.e.,

\[T=T_n(\ul{a}\dd\ul{b})=\left(T_n(\ul{a})\cup T_n(\ul{b})\right)\setminus\left(T_n(\ul{a})\cap T_n(\ul{b})\right).\]
\begin{remark}
Note that if $\sum a_i \geq \sum b_i$, then $T=T_n(\ul{a}\dd\ul{b})=T_n(\ul{b})\setminus T_n(\ul{a})$.  For convenience, we will assume $\sum a_i \geq \sum b_i$ for the remainder of this paper.  
\end{remark}
\begin{example}\label{typeCexample} The type-C meander $M_{12}^\C((2,1,2,6) \dd (3,2,1,2))$ has tail $T=\{v_9,v_{10},v_{11}\}$, indicated by yellow vertices in Figure \ref{CMeander}.

\begin{figure}[H]
\[\begin{tikzpicture}[scale=.67]
\vertex (1) at (1,0) {1};
\vertex (2) at (2,0) {2};
\vertex (3) at (3,0) {3};
\vertex (4) at (4,0) {4};
\vertex (5) at (5,0) {5};
\vertex (6) at (6,0) {6};
\vertex (7) at (7,0) {7};
\vertex (8) at (8,0) {8};
\vertex[fill=yellow] (9) at (9,0) {9};
\vertex[fill=yellow] (10) at (10,0) {10};
\vertex[fill=yellow] (11) at (11,0) {11};
\vertex (12) at (12,0) {12};

\path  
(1) edge[bend left=50] (2)
(4) edge[bend left=50] (5)
(6) edge[bend left=50] (11)
(7) edge[bend left=50] (10)
(8) edge[bend left=50] (9)
(1) edge[bend right=50] (3)
(4) edge[bend right=50] (5)
(7) edge[bend right=50] (8)
;\end{tikzpicture}\]
\caption{The meander $M_{12}^\C((2,1,2,6) \dd (3,2,1,2))$}
\label{CMeander}
\end{figure}
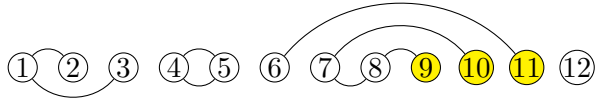

%\noindent
%Since both strings $(2,1,2,6)$ and $(3,2,1,2)$ have sum less than 12, we can easily obtain the graph $M_{11}^\C((2,1,2,6) \dd (3,2,1,2))$ 
%from the graph $M_{12}^\C((2,1,2,6) \dd (3,2,1,2))$ by removing vertex 12,
%and the tail remains the same.
\end{example}

The following theorem is the type-C analogue of the combinatorial formula for the index of type-A seaweeds given in Theorem \ref{DKformula}.

\begin{thm}[Coll et al. \textbf{\cite{Coll4}}, Theorem 4.5]\label{symplectic index}
Consider the seaweed $\mf{p}_n^\C(\ul{a} \dd \ul{b})$, and let $T=T_n(\ul{a}\dd\ul{b})$.  The index of $\mf{p}_n^\C(\ul{a} \dd \ul{b})$ is equal to $2C + \widetilde{P}$ where $C$ is the number of cycles in $M_n^\C(\ul{a} \dd \ul{b})$ and $\widetilde{P}$ is the number of connected
components containing either zero or two vertices from $T$ in $M_n^\C(\ul{a} \dd \ul{b})$.

\end{thm}

\begin{example}
In Example \ref{typeCexample}, $C = 1$ and $\widetilde{P} = 3$, so $\ind \mf{p}_{12}^\C((2,1,2,6) \dd (3,2,1,2))=5$.
\end{example} 

The tail allows us to completely classify Frobenius type-C seaweeds up to similarity.  The combinatorial formula in Theorem \ref{symplectic index} is zero when $C$ and $\widetilde{P}$ are both zero, i.e., when all components of the meander are paths with one end in the tail.  We record this in the following corollary. 

\begin{cor}\label{TypeCForest}
A type-C seaweed is Frobenius if and only if its corresponding meander is a forest rooted in the tail.  
\end{cor} 

\begin{example}\label{TypeCFrobeniusExample} The seaweed $\mf{p}_{14}^\C\frac{7 \dd 7}{11}$ is Frobenius by Corollary \ref{TypeCForest}. Below is its meander, with components highlighted.

\begin{figure}[H]
\[\begin{tikzpicture}[scale=.67]
\vertex (1) at (1,0) {1};
\vertex (2) at (2,0) {2};
\vertex (3) at (3,0) {3};
\vertex (4) at (4,0) {4};
\vertex (5) at (5,0) {5};
\vertex (6) at (6,0) {6};
\vertex (7) at (7,0) {7};
\vertex (8) at (8,0) {8};
\vertex (9) at (9,0) {9};
\vertex (10) at (10,0) {10};
\vertex (11) at (11,0) {11};
\vertex[fill=yellow] (12) at (12,0) {12};
\vertex[fill=yellow] (13) at (13,0) {13};
\vertex[fill=yellow] (14) at (14,0) {14};

\path  
(1) edge[bend left=50, color=red, line width=1.2pt] (7)
(2) edge[bend left=50, color=blue, line width=1.2pt] (6)
(3) edge[bend left=50, color=red, line width=1.2pt] (5)
(8) edge[bend left=50, color=black, line width=1.2pt] (14)
(9) edge[bend left=50, color=red, line width=1.2pt] (13)
(10) edge[bend left=50, color=blue, line width=1.2pt] (12)
(1) edge[bend right=50, color=red, line width=1.2pt] (11)
(2) edge[bend right=50, color=blue, line width=1.2pt] (10)
(3) edge[bend right=50, color=red, line width=1.2pt] (9)
(4) edge[bend right=50, color=black, line width=1.2pt] (8)
(5) edge[bend right=50, color=red, line width=1.2pt] (7)
;\end{tikzpicture}\]
\caption{The meander $M_{14}^\C\frac{7 \dd 7}{11}$ with components highlighted}
\label{CFrobExample} 
\end{figure}
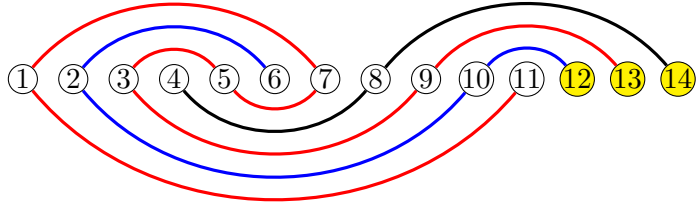
\end{example}
\noindent 
The following corollary gives a necessary condition for a symplectic seaweed to have minimal index. 

\begin{cor}[Coll et al. \textbf{\cite{Coll4}}, Corollary 4.7]\label{c necessary}
If $\ind \mf{p}_n^\C(\ul{a} \dd \ul{b})=0$, then
$\sum a_i=n$, and $\sum b_i=n-r<n$, and there must be exactly $r$ odd integers among $\ul{a}$ and $\ul{b}$.

\end{cor}

\subsection{Type-C index formulas}

Next, we consider type-C seaweeds where $\ul{a}$ and $\ul{b}$ have a small number of parts.
Theorem 5.5 in \textbf{\cite{Panyushev1}} covers all cases when either $\ul{a}=\emptyset$ or $\ul{b}=\emptyset$. The
next theorem considers the case when each of $\ul{a}$ and $\ul{b}$ has one part and contains a corrected typo from \textbf{\cite{Coll4}}.

\begin{cor}[Coll et al. \textbf{\cite{Coll4}}, Corollary 5.1]\label{CTwoBlock} 
If $a=b$, then $\ind \mf{p}_n^\C((a) \dd (b))=n$.
Otherwise,

\[\ind \mf{p}_n^\C((a) \dd (b))=
\begin{cases}
n-a+\lf\frac{a-b}{2}\rf, & \text{ if }a-b\text{ is even;}\\
n-a+\lf\frac{a-b-1}{2}\rf, & \text{ if }a-b\text{ is odd}.
\end{cases}\]

\end{cor}

%\begin{proof}

%If $n>a$ then $\ind \mf{p}_n^\C((a) \dd (b))=n-a+\ind \mf{p}_{n-a}^\C((a) \dd (b))$, so it suffices to show that
%\begin{equation}\label{2 parts eq}
%\ind \mf{p}_n^\C((n) \dd (b))=
%\begin{cases}
%\lf\frac{n-b}{2}\rf & \text{ if }n\text{ is even}\\
%\lf\frac{n-b-1}{2}\rf & \text{ if }n\text{ is odd}.
%\end{cases}
%\end{equation}

%Suppose $b\leq n/2$. By Theorem \ref{C inductive} and Theorem \ref{C parabolic} we have
%\[\ind \mf{p}_n^\C((n) \dd (b))=\ind \mf{p}_n^\C((b) \dd (n))=\ind \mf{p}_{n-b}^\C((\emptyset) \dd (n-2b,b))
%=\lf\frac{n-2b}{2}\rf+\lf\frac{b}{2}\rf.\]
%If $n$ is even then $n-2b$ is even and
%\[\lf\frac{n-2b}{2}\rf+\lf\frac{b}{2}\rf=\frac{n-2b}{2}+\lf\frac{b}{2}\rf=\lf\frac{n-b}{2}\rf.\]
%If $n$ is odd then $n-2b$ is odd and
%\[\lf\frac{n-2b}{2}\rf+\lf\frac{b}{2}\rf=\frac{n-2b-1}{2}+\lf\frac{b}{2}\rf=\lf\frac{n-b-1}{2}\rf.\]

%Suppose $b>n/2$. We prove \eqref{2 parts eq} by induction on $n$. The base case is trivial. For the inductive step, 
%use Theorem \ref{C inductive}:
%\[\ind \mf{p}_n^\C((n) \dd (b))=\ind \mf{p}_n^\C((b) \dd (n))=\ind \mf{p}_{b}^\C((2b-n) \dd (b)).\]
%Using the inductive hypothesis, it is easy to show that \eqref{2 parts eq} holds by considering all four cases 
%for the parity of $n$ and $b$.

%\end{proof}

Now, we consider when $\ul{a}$ and $\ul{b}$ have a total of four parts. The following remark illustrates why we need not consider more complicated block configurations.

\begin{remark}
Consider the type-C meander 
$M^C_n ((a,b,c)\dd (d))$ where $a+b+c=n$ and $d<n$.  The index computations for this meander are analogous to those for the meander $M^A_{2n-d} ((n-d,a,b,c)\dd (2n-d))$ which by Theorem \ref{5 parts} has not just no linear gcd formula, but no polynomial gcd formula for its index. 
\end{remark}

We have the following three theorems for when $\ul{a}$ and $\ul{b}$ have a total of three parts.  

\begin{thm}[Coll et al. \textbf{\cite{Coll4}}, Theorm 5.2]\label{3 parts thm1}
Let $a+b=n$. If $c=n-1$ or $c=n-2$, then
\begin{eqnarray}
\ind \mf{p}_n^\C((a,b) \dd (c))=\gcd(a+b,b+c)-1.
\end{eqnarray}
\end{thm}

\begin{thm}[Coll et al. \textbf{\cite{Coll4}}, Theorem 5.3]\label{3 parts thm2}
If $a+b=n$, then $\ind \mf{p}_n^\C((a,b) \dd (c))=0$ if and only if one of the following conditions hold:
\begin{enumerate}[\textup(i\textup)]
\item $c=n-1$ and $\gcd(a+b,b+c)=1,$
\item $c=n-2$ and $\gcd(a+b,b+c)=1,$
\item $c=n-3$, the integers $a,b,$ and $c$ are all odd, and $\gcd(a+b,b+c)=2$.
\end{enumerate}
\end{thm}

\begin{thm}[Coll et al. \textbf{\cite{Coll4}}, Theorem 5.6]\label{3 parts thm4}
The index of $\mf{p}_n^\C((n) \dd (a,b))$ is equal to zero 
if and only if one of the following conditions hold:
\begin{enumerate}[\textup(i\textup)]
\item $a+b=n-1$ and $\gcd(a+b,b+1)=1$,

\item $a+b=n-2$ and $\gcd(a+b,b+2)=1$,

\item $a+b=n-3$, the integers $n,a,$ and $b$ are all odd, and $\gcd(a+b,b+3)=2$.

\end{enumerate}
\end{thm}

\begin{remark}
 We could have alternatively used Theorem \ref{3 parts thm2} or \ref{3 parts thm4} 
\textup(instead of Corollary \ref{TypeCForest}\textup)
to conclude that the seaweed in Figure \ref{CFrobExample} is Frobenius.  
\end{remark}

\section{Type D - $\mathfrak{so}(2n)$}
Type-D seaweeds share some similarities with seaweeds in the types-B and C cases. However, 
in their standard representations, type-D seaweeds do not necessarily have seaweed shape.  
In Section 5.2, we discern what configuration of the defining roots causes this 
(see Theorem \ref{seaweed shape}).  When type-D seaweeds \textit{do} have seaweed shape, their shape is exactly the same as in type C.  
In Section 5.3, to deal with seaweed-shaped seaweeds, we introduce (as with types B and C) the notion of a type-D meander and tail.   In Section 5.4.1, we develop the type-D analogue of the combinatorial $2C +\widetilde{P}$ index formula (see Theorem \ref{typeD}). In Section 5.4.2, we leverage this combinatorics to yield linear gcd conditions for the index of certain type-D seaweeds based on the
sizes of the parts that define them (see Theorem \ref{HomotopyTypeH3}). In particular, we characterize Frobenius type-D seaweed-shaped seaweeds based on linear gcd conditions (and congruence properties) in the sizes of the parts that define the seaweed (see Theorems \ref{HomotopyTypeH1} and \ref{HomotopyTypeH11}). We conclude the analysis of type-D seaweed-shaped seaweeds by establishing the type-D analogue of Theorem \ref{5 parts} of Karnauhova and Liebscher (see Theorem \ref{KarLiebTypeD}).  
In Section 5.5, we consider type-D seaweeds which do not have seaweed shape.  We show that the index of a seaweed without seaweed shape can be computed by considering a seaweed which \textit{does} have seaweed shape, and the index of the former and the latter differ by a constant (either 0 or 2).  We use this to provide a classification of Frobenius type-D seaweeds which do not have seaweed shape (see Theorem \ref{FrobeniusWithoutSeaweedShape}).

\subsection{Type-D seaweeds}
Let $\mf so(2n)$ be the algebra of matrices with the following block form

\[\mf so(2n)=\left\{\begin{bmatrix} A & B \\ C & -\wh{A}\end{bmatrix}: B=-\wh{B}, C=-\wh{C} \right\},\]
where $A, B,$ and $C$ are $n\times n$ matrices and $\wh{A}$ is the transpose of 
$A$ with respect to the antidiagonal.
Choose the same triangular decomposition as was done in the $\mf{sl}(n)$ and $\mf{sp}(2n)$ cases,
that is $\mf{so}(2n)=\mf{u_+}\oplus\mf{h}\oplus\mf{u_-}$.
Let $\Pi=\{\alpha_1,\dots ,\alpha_{n}\}$ denote its set of simple roots,
where $\alpha_n$ is the exceptional root.
Let $\mf{p}_n^\D(\Psi \dd \Psi')$ denote a seaweed subalgebra where
$\Psi$ and $\Psi'$ are subsets of $\Pi$. 
We find it convenient to visualize the seaweed by picturing the omitted roots.  We call this the \textit{split Dynkin diagram} for a seaweed.  See Figure \ref{DSplitDynkinDiagram}.  

\begin{figure}[H]

\[\begin{tikzpicture}[scale=.79]

\draw (1,3) node[draw,circle,fill=black,minimum size=5pt,inner sep=0pt] (1+) {};
\draw (2,3) node[draw,circle,fill=black,minimum size=5pt,inner sep=0pt] (2+) {};
\draw (3,3) node[draw,circle,fill=white,minimum size=5pt,inner sep=0pt] (3+) {};
\draw (4,3) node[draw,circle,fill=black,minimum size=5pt,inner sep=0pt] (4+) {};
\draw (5,3) node[draw,circle,fill=black,minimum size=5pt,inner sep=0pt] (5+) {};
\draw (6,3) node[draw,circle,fill=black,minimum size=5pt,inner sep=0pt] (6+) {};
\draw (7,3.7) node[draw,circle,fill=black,minimum size=5pt,inner sep=0pt] (7+) {};
\draw (7,2.3) node[draw,circle,fill=black,minimum size=5pt,inner sep=0pt] (8+) {};

\draw (1,0) node[draw,circle,fill=black,minimum size=5pt,inner sep=0pt] (1-) {};
\draw (2,0) node[draw,circle,fill=black,minimum size=5pt,inner sep=0pt] (2-) {};
\draw (3,0) node[draw,circle,fill=black,minimum size=5pt,inner sep=0pt] (3-) {};
\draw (4,0) node[draw,circle,fill=white,minimum size=5pt,inner sep=0pt] (4-) {};
\draw (5,0) node[draw,circle,fill=black,minimum size=5pt,inner sep=0pt] (5-) {};
\draw (6,0) node[draw,circle,fill=black,minimum size=5pt,inner sep=0pt] (6-) {};
\draw (7,.7) node[draw,circle,fill=black,minimum size=5pt,inner sep=0pt] (7-) {};
\draw (7,-.7) node[draw,circle,fill=white,minimum size=5pt,inner sep=0pt] (8-) {};

\draw (1-) to (2-);
\draw (2-) to (3-);
\draw (5-) to (6-);
\draw (6-) to (7-);
\draw (1+) to (2+);
\draw (4+) to (5+);
\draw (5+) to (6+);
\draw (6+) to (7+);
\draw (6+) to (8+);

\node at (1,2.5) {$\alpha_1$};
\node at (2,2.5) {$\alpha_2$};
\node at (3,2.5) {$\alpha_3$};
\node at (4,2.5) {$\alpha_4$};
\node at (5,2.5) {$\alpha_5$};
\node at (6,2.5) {$\alpha_6$};
\node at (7.5,3.7) {$\alpha_7$};
\node at (7.5,2.3) {$\alpha_8$};

\node at (1,.5) {$\alpha_1$};
\node at (2,.5) {$\alpha_2$};
\node at (3,.5) {$\alpha_3$};
\node at (4,.5) {$\alpha_4$};
\node at (5,.5) {$\alpha_5$};
\node at (6,.5) {$\alpha_6$};
\node at (7.5,.7) {$\alpha_7$};
\node at (7.5,-.7) {$\alpha_8$};

;\end{tikzpicture}\]
\caption{The seaweed $\mf{p}_8^\D(\{\alpha_1, \alpha_2, \alpha_4, \alpha_5, \alpha_6, \alpha_7, \alpha_8\}\dd \{\alpha_1, \alpha_2, \alpha_3, \alpha_4, \alpha_5, \alpha_6, \alpha_7\})$}
\label{DSplitDynkinDiagram}
\end{figure}
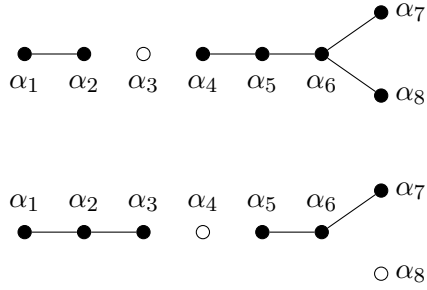

\subsection{Type-D seaweed-shaped seaweeds}

%In the types A, B, and C cases, we defined seaweeds by compositions or partial compositions of $n$.  
%They could also be defined by subsets of simple roots included in each parabolic, from which compositions could be obtained.  
%To study the type-D case, we will eventually be able to use partial compositions of $n$, but first we must analyze which subsets of simple roots define seaweeds without seaweed shape.  
%Once seaweeds without seaweed shape are classified, we may assign pairs of partial compositions of $n$ to each seaweed with seaweed shape.  

Curiously, type-D seaweeds do not necessarily have seaweed shape in their natural representation.  
Consequently, not all type-D seaweeds have the block triangular form which compositions can be obtained.  
We determine which type-D seaweeds do not have seaweed shape (see Theorem \ref{seaweed shape}).  We first examine type-D parabolics.  

If $\mf{p}$ is a parabolic subalgebra of $\mf{so}(2n)$ defined by $\Psi \subseteq \Pi$ with $\alpha_{n-1} \not\in \Psi$ and $\alpha_n \in \Psi$, then $\mf{p}$ does not have seaweed shape.  However, if we remove $\alpha_n$ from and adjoin $\alpha_{n-1}$ to $\Psi$, then this yields an isomorphic parabolic which does have seaweed shape.  See Figure \ref{DParabolicWithoutSeaweedShape}.

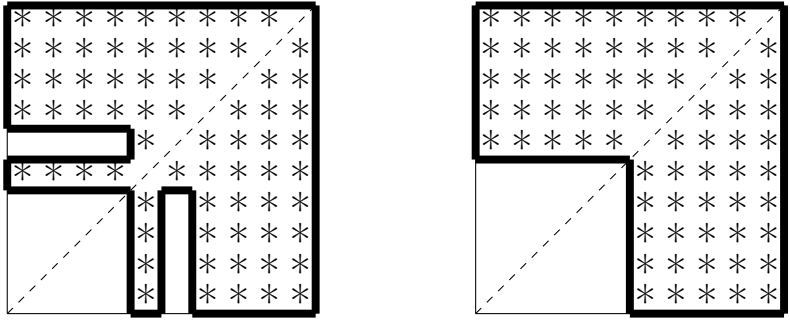
\begin{figure}[H]
\[\begin{tikzpicture}[scale=.41]
\draw (0,0) -- (0,10);
\draw (0,10) -- (10,10);
\draw (10,10) -- (10,0);
\draw (10,0) -- (0,0);

\draw [line width=3](0,10) -- (10,10);
\draw [line width=3](10,10) -- (10,0);
\draw [line width=3](0,10) -- (0,6);
\draw [line width=3](0,6) -- (4,6);
\draw [line width=3](4,6) -- (4,5);
\draw [line width=3](4,5) -- (0,5);
\draw [line width=3](0,5) -- (0,4);
\draw [line width=3](0,4) -- (4,4);
\draw [line width=3](4,4) -- (4,0);
\draw [line width=3](4,0) -- (5,0);
\draw [line width=3](5,0) -- (5,4);
\draw [line width=3](5,4) -- (6,4);
\draw [line width=3](6,4) -- (6,0);
\draw [line width=3](6,0) -- (10,0);

\draw [dashed] (0,0) -- (10,10);

\node at (0.5,9.4) {{\LARGE *}};
\node at (1.5,9.4) {{\LARGE *}};
\node at (2.5,9.4) {{\LARGE *}};
\node at (3.5,9.4) {{\LARGE *}};
\node at (4.5,9.4) {{\LARGE *}};
\node at (5.5,9.4) {{\LARGE *}};
\node at (6.5,9.4) {{\LARGE *}};
\node at (7.5,9.4) {{\LARGE *}};
\node at (8.5,9.4) {{\LARGE *}};

\node at (0.5,8.4) {{\LARGE *}};
\node at (1.5,8.4) {{\LARGE *}};
\node at (2.5,8.4) {{\LARGE *}};
\node at (3.5,8.4) {{\LARGE *}};
\node at (4.5,8.4) {{\LARGE *}};
\node at (5.5,8.4) {{\LARGE *}};
\node at (6.5,8.4) {{\LARGE *}};
\node at (7.5,8.4) {{\LARGE *}};
\node at (9.5,8.4) {{\LARGE *}};

\node at (0.5,7.4) {{\LARGE *}};
\node at (1.5,7.4) {{\LARGE *}};
\node at (2.5,7.4) {{\LARGE *}};
\node at (3.5,7.4) {{\LARGE *}};
\node at (4.5,7.4) {{\LARGE *}};
\node at (5.5,7.4) {{\LARGE *}};
\node at (6.5,7.4) {{\LARGE *}};
\node at (8.5,7.4) {{\LARGE *}};
\node at (9.5,7.4) {{\LARGE *}};

\node at (0.5,6.4) {{\LARGE *}};
\node at (1.5,6.4) {{\LARGE *}};
\node at (2.5,6.4) {{\LARGE *}};
\node at (3.5,6.4) {{\LARGE *}};
\node at (4.5,6.4) {{\LARGE *}};
\node at (5.5,6.4) {{\LARGE *}};
\node at (7.5,6.4) {{\LARGE *}};
\node at (8.5,6.4) {{\LARGE *}};
\node at (9.5,6.4) {{\LARGE *}};

\node at (4.5,5.4) {{\LARGE *}};
\node at (6.5,5.4) {{\LARGE *}};
\node at (7.5,5.4) {{\LARGE *}};
\node at (8.5,5.4) {{\LARGE *}};
\node at (9.5,5.4) {{\LARGE *}};

\node at (0.5,4.4) {{\LARGE *}};
\node at (1.5,4.4) {{\LARGE *}};
\node at (2.5,4.4) {{\LARGE *}};
\node at (3.5,4.4) {{\LARGE *}};
\node at (5.5,4.4) {{\LARGE *}};
\node at (6.5,4.4) {{\LARGE *}};
\node at (7.5,4.4) {{\LARGE *}};
\node at (8.5,4.4) {{\LARGE *}};
\node at (9.5,4.4) {{\LARGE *}};

\node at (4.5,3.4) {{\LARGE *}};
\node at (6.5,3.4) {{\LARGE *}};
\node at (7.5,3.4) {{\LARGE *}};
\node at (8.5,3.4) {{\LARGE *}};
\node at (9.5,3.4) {{\LARGE *}};

\node at (4.5,2.4) {{\LARGE *}};
\node at (6.5,2.4) {{\LARGE *}};
\node at (7.5,2.4) {{\LARGE *}};
\node at (8.5,2.4) {{\LARGE *}};
\node at (9.5,2.4) {{\LARGE *}};

\node at (4.5,1.4) {{\LARGE *}};
\node at (6.5,1.4) {{\LARGE *}};
\node at (7.5,1.4) {{\LARGE *}};
\node at (8.5,1.4) {{\LARGE *}};
\node at (9.5,1.4) {{\LARGE *}};

\node at (4.5,.4) {{\LARGE *}};
\node at (6.5,.4) {{\LARGE *}};
\node at (7.5,.4) {{\LARGE *}};
\node at (8.5,.4) {{\LARGE *}};
\node at (9.5,.4) {{\LARGE *}};

\end{tikzpicture}
\hspace{2cm}
\begin{tikzpicture}[scale=.41]
\draw (0,0) -- (0,10);
\draw (0,10) -- (10,10);
\draw (10,10) -- (10,0);
\draw (10,0) -- (0,0);

\draw [line width=3](0,10) -- (10,10);
\draw [line width=3](10,10) -- (10,0);
\draw [line width=3](0,10) -- (0,5);
\draw [line width=3](0,5) -- (5,5);
\draw [line width=3](5,5) -- (5,0);
\draw [line width=3](5,0) -- (10,0);

\draw [dashed] (0,0) -- (10,10);

\node at (0.5,9.4) {{\LARGE *}};
\node at (1.5,9.4) {{\LARGE *}};
\node at (2.5,9.4) {{\LARGE *}};
\node at (3.5,9.4) {{\LARGE *}};
\node at (4.5,9.4) {{\LARGE *}};
\node at (5.5,9.4) {{\LARGE *}};
\node at (6.5,9.4) {{\LARGE *}};
\node at (7.5,9.4) {{\LARGE *}};
\node at (8.5,9.4) {{\LARGE *}};

\node at (0.5,8.4) {{\LARGE *}};
\node at (1.5,8.4) {{\LARGE *}};
\node at (2.5,8.4) {{\LARGE *}};
\node at (3.5,8.4) {{\LARGE *}};
\node at (4.5,8.4) {{\LARGE *}};
\node at (5.5,8.4) {{\LARGE *}};
\node at (6.5,8.4) {{\LARGE *}};
\node at (7.5,8.4) {{\LARGE *}};
\node at (9.5,8.4) {{\LARGE *}};

\node at (0.5,7.4) {{\LARGE *}};
\node at (1.5,7.4) {{\LARGE *}};
\node at (2.5,7.4) {{\LARGE *}};
\node at (3.5,7.4) {{\LARGE *}};
\node at (4.5,7.4) {{\LARGE *}};
\node at (5.5,7.4) {{\LARGE *}};
\node at (6.5,7.4) {{\LARGE *}};
\node at (8.5,7.4) {{\LARGE *}};
\node at (9.5,7.4) {{\LARGE *}};

\node at (0.5,6.4) {{\LARGE *}};
\node at (1.5,6.4) {{\LARGE *}};
\node at (2.5,6.4) {{\LARGE *}};
\node at (3.5,6.4) {{\LARGE *}};
\node at (4.5,6.4) {{\LARGE *}};
\node at (5.5,6.4) {{\LARGE *}};
\node at (7.5,6.4) {{\LARGE *}};
\node at (8.5,6.4) {{\LARGE *}};
\node at (9.5,6.4) {{\LARGE *}};

\node at (0.5,5.4) {{\LARGE *}};
\node at (1.5,5.4) {{\LARGE *}};
\node at (2.5,5.4) {{\LARGE *}};
\node at (3.5,5.4) {{\LARGE *}};
\node at (4.5,5.4) {{\LARGE *}};
\node at (6.5,5.4) {{\LARGE *}};
\node at (7.5,5.4) {{\LARGE *}};
\node at (8.5,5.4) {{\LARGE *}};
\node at (9.5,5.4) {{\LARGE *}};

\node at (5.5,4.4) {{\LARGE *}};
\node at (6.5,4.4) {{\LARGE *}};
\node at (7.5,4.4) {{\LARGE *}};
\node at (8.5,4.4) {{\LARGE *}};
\node at (9.5,4.4) {{\LARGE *}};

\node at (5.5,3.4) {{\LARGE *}};
\node at (6.5,3.4) {{\LARGE *}};
\node at (7.5,3.4) {{\LARGE *}};
\node at (8.5,3.4) {{\LARGE *}};
\node at (9.5,3.4) {{\LARGE *}};

\node at (5.5,2.4) {{\LARGE *}};
\node at (6.5,2.4) {{\LARGE *}};
\node at (7.5,2.4) {{\LARGE *}};
\node at (8.5,2.4) {{\LARGE *}};
\node at (9.5,2.4) {{\LARGE *}};

\node at (5.5,1.4) {{\LARGE *}};
\node at (6.5,1.4) {{\LARGE *}};
\node at (7.5,1.4) {{\LARGE *}};
\node at (8.5,1.4) {{\LARGE *}};
\node at (9.5,1.4) {{\LARGE *}};

\node at (5.5,.4) {{\LARGE *}};
\node at (6.5,.4) {{\LARGE *}};
\node at (7.5,.4) {{\LARGE *}};
\node at (8.5,.4) {{\LARGE *}};
\node at (9.5,.4) {{\LARGE *}};

\end{tikzpicture}
\]

\caption{The parabolic of $\mf{so}(10)$ with $\Psi = \{\alpha_1, \alpha_2, \alpha_3, \alpha_5\}$ (left) does not have seaweed shape, while the parabolic of $\mf{so}(10)$ with $\Psi = \{\alpha_1, \alpha_2, \alpha_3, \alpha_4\}$ (right) does have seaweed shape. }
\label{DParabolicWithoutSeaweedShape}
\end{figure}

All other type-D parabolics have seaweed shape.  However, making this type of ``switch" does not help for all type-D seaweeds, i.e., pairs of parabolics.  In particular, the following theorem classifies the type-D seaweeds without seaweed shape.

\begin{thm}\label{seaweed shape}
Without loss of generality, $\mf{p}_n^\D(\Psi \dd \Psi')$ does not have seaweed shape if and only if $\alpha_{n-1} \in \Psi \setminus \Psi '$ and $\alpha_n \in \Psi ' \setminus \Psi$.  
\end{thm}

\proof 

Let $\mf{p} = \mf{p}_n^\D (\Psi \dd \Psi ')$ with $\alpha_{n-1} \in \Psi \setminus \Psi '$ and $\alpha_n \in \Psi ' \setminus \Psi$.  Let $I = \max\{i \ | \ \alpha_i \in \Psi\}$.  If $\Psi$ is empty, set $I = 0$.  The root spaces corresponding to $\alpha_{I+j} + \alpha_{I+j+1} + ... + \alpha_{n-1}$ for $j = 1, ..., n-I-2$ are in $\mf{p}_1$, and the root spaces corresponding to $\alpha_{I+j} + \alpha_{I+j+1} + ... + \alpha_{n-2} + \alpha_n$ for $j = 1, ..., n-I-2$ are in $\mf{p}_2$, but those corresponding to $\alpha_{I+j} + \alpha_{I+j+1} + ... + \alpha_{n-1}$ for $j = 1, ..., n-I-2$ are not.  So $\mf{p}$ does not have seaweed shape.  

For the converse, we make the following observation.  
Let $\psi \subseteq \{\alpha_1, ... \alpha_{n-2}\}$.  Let $\Psi = \psi + \alpha_{n-1}$, and let $\Psi ' = \psi + \alpha_n$.  Let $\mf{p}_1$ be the parabolic subalgebra of $\mf{so}(2n)$ determined by $\Pi \setminus \Psi$, and let $\mf{p}_2$ be the parabolic subalgebra of $\mf{so}(2n)$ determined by $\Pi \setminus \Psi'$.  Then $\mf{p}_1 \cong \mf{p}_2$.  The result follows. \qed

The following figure shows a type-D seaweed without seaweed shape.

\begin{figure}[H]
\[\begin{tikzpicture}[scale=.41]
\draw (0,0) -- (0,10);
\draw (0,10) -- (10,10);
\draw (10,10) -- (10,0);
\draw (10,0) -- (0,0);

\draw [line width=3](0,10) -- (2,10);
\draw [line width=3](2,10) -- (2,8);
\draw [line width=3](2,8) -- (5,8);
\draw [line width=3](5,8) -- (5,5);
\draw [line width=3](5,5) -- (8,5);
\draw [line width=3](8,5) -- (8,2);
\draw [line width=3](8,2) -- (10,2);
\draw [line width=3](10,2) -- (10,0);

\draw [line width=3](0,10) -- (0,9);
\draw [line width=3](0,9) -- (1,9);
\draw [line width=3](1,9) -- (1,7);
\draw [line width=3](1,6) -- (3,6);
\draw [line width=3](1,7) -- (1,6);
\draw [line width=3](3,6) -- (4,6);
\draw [line width=3](4,6) -- (4,5);
\draw [line width=3](4,5) -- (5,5);
\draw [line width=3](5,5) -- (5,4);
\draw [line width=3](5,4) -- (6,4);
\draw [line width=3](6,4) -- (6,3);
\draw [line width=3](6,3) -- (6,1);
\draw [line width=3](6,1) -- (7,1);
\draw [line width=3](7,1) -- (9,1);
\draw [line width=3](9,1) -- (9,0);
\draw [line width=3](9,0) -- (10,0);
\draw [line width=3](1,4) -- (1,5);
\draw [line width=3](1,5) -- (4,5);
\draw [line width=3](4,5) -- (4,4);
\draw [line width=3](4,4) -- (1,4);
\draw [line width=3](4,1) -- (5,1);
\draw [line width=3](5,1) -- (5,4);
\draw [line width=3](5,4) -- (4,4);
\draw [line width=3](4,4) -- (4,1);

\draw [dashed] (0,0) -- (10,10);

\node at (0.5,9.4) {{\LARGE *}};
\node at (1.5,9.4) {{\LARGE *}};
\node at (1.5,8.4) {{\LARGE *}};
\node at (1.5,7.4) {{\LARGE *}};
\node at (2.5,7.4) {{\LARGE *}};
\node at (3.5,7.4) {{\LARGE *}};
\node at (4.5,7.4) {{\LARGE *}};
\node at (1.5,6.4) {{\LARGE *}};
\node at (2.5,6.4) {{\LARGE *}};
\node at (3.5,6.4) {{\LARGE *}};
\node at (4.5,6.4) {{\LARGE *}};
\node at (4.5,5.4) {{\LARGE *}};
\node at (1.5,4.4) {{\LARGE *}};
\node at (2.5,4.4) {{\LARGE *}};
\node at (3.5,4.4) {{\LARGE *}};
\node at (5.5,4.4) {{\LARGE *}};
\node at (6.5,4.4) {{\LARGE *}};
\node at (7.5,4.4) {{\LARGE *}};
\node at (4.5,3.4) {{\LARGE *}};
\node at (6.5,3.4) {{\LARGE *}};
\node at (7.5,3.4) {{\LARGE *}};
\node at (4.5,2.4) {{\LARGE *}};
\node at (6.5,2.4) {{\LARGE *}};
\node at (7.5,2.4) {{\LARGE *}};
\node at (4.5,1.4) {{\LARGE *}};
\node at (6.5,1.4) {{\LARGE *}};
\node at (7.5,1.4) {{\LARGE *}};
\node at (8.5,1.4) {{\LARGE *}};
\node at (9.5,1.4) {{\LARGE *}};
\node at (9.5,0.4) {{\LARGE *}};

\end{tikzpicture}
\]

\caption{ $\mf{p}_5^\D(
\{\alpha_2,\alpha_3,\alpha_5\} \dd \{\alpha_1,\alpha_3,\alpha_4\})$}
\label{DSeaweedWithoutSeaweedShape}
\end{figure}
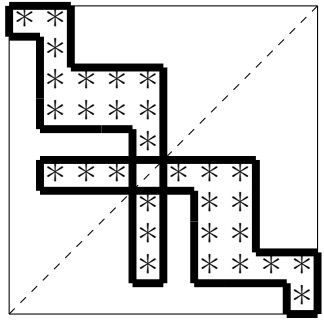

\begin{remark}
There is a nice visual representation for when a type-D seaweed does not have seaweed shape using a split Dynkin diagram: any of $\alpha_1, ..., \alpha_{n-2}$ can be included in either parabolic, indicated by the gray vertices, but the essential features occur are the bifurcation points.  
\end{remark}

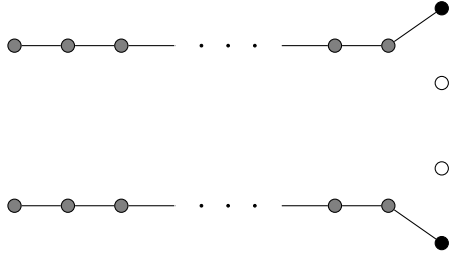
\begin{figure}[H]

\[\begin{tikzpicture}[scale=.71]

\draw (1,3.75) node[draw,circle,fill=gray,minimum size=5pt,inner sep=0pt] (1+) {};
\draw (2,3.75) node[draw,circle,fill=gray,minimum size=5pt,inner sep=0pt] (2+) {};
\draw (3,3.75) node[draw,circle,fill=gray,minimum size=5pt,inner sep=0pt] (3+) {};
\draw (4,3.75) node[draw,circle,fill=black,minimum size=0pt,inner sep=0pt] (4+) {};
\draw (4.5,3.75) node[draw,circle,fill=black,minimum size=1pt,inner sep=0pt] {};
\draw (5,3.75) node[draw,circle,fill=black,minimum size=1pt,inner sep=0pt] {};
\draw (5.5,3.75) node[draw,circle,fill=black,minimum size=1pt,inner sep=0pt] {};
\draw (6,3.75) node[draw,circle,fill=black,minimum size=0pt,inner sep=0pt] (6+) {};
\draw (7,3.75) node[draw,circle,fill=gray,minimum size=5pt,inner sep=0pt] (7+) {};
\draw (8,3.75) node[draw,circle,fill=gray,minimum size=5pt,inner sep=0pt] (8+) {};
\draw (9,4.45) node[draw,circle,fill=black,minimum size=5pt,inner sep=0pt] (9+) {};
\draw (9,3.05) node[draw,circle,fill=white,minimum size=5pt,inner sep=0pt] (10+) {};

\draw (1+) to (2+);
\draw (2+) to (3+);
\draw (3+) to (4+);
\draw (6+) to (7+);
\draw (7+) to (8+);
\draw (8+) to (9+);

%\node at (1,3) {$\alpha_1$};
%\node at (2,3) {$\alpha_2$};
%\node at (3,3) {$\alpha_3$};
%\node at (7,3) {$\alpha_{n-3}$};
%\node at (8,3) {$\alpha_{n-2}$};
%\node at (9.7,4.45) {$\alpha_{n-1}$};
%\node at (9.5,3.05) {$\alpha_n$};

\draw (1,.75) node[draw,circle,fill=gray,minimum size=5pt,inner sep=0pt] (1-) {};
\draw (2,.75) node[draw,circle,fill=gray,minimum size=5pt,inner sep=0pt] (2-) {};
\draw (3,.75) node[draw,circle,fill=gray,minimum size=5pt,inner sep=0pt] (3-) {};
\draw (4,.75) node[draw,circle,fill=black,minimum size=0pt,inner sep=0pt] (4-) {};
\draw (4.5,.75) node[draw,circle,fill=black,minimum size=1pt,inner sep=0pt] {};
\draw (5,.75) node[draw,circle,fill=black,minimum size=1pt,inner sep=0pt] {};
\draw (5.5,.75) node[draw,circle,fill=black,minimum size=1pt,inner sep=0pt] {};
\draw (6,.75) node[draw,circle,fill=black,minimum size=0pt,inner sep=0pt] (6-) {};
\draw (7,.75) node[draw,circle,fill=gray,minimum size=5pt,inner sep=0pt] (7-) {};
\draw (8,.75) node[draw,circle,fill=gray,minimum size=5pt,inner sep=0pt] (8-) {};
\draw (9,1.45) node[draw,circle,fill=white,minimum size=5pt,inner sep=0pt] (9-) {};
\draw (9,.05) node[draw,circle,fill=black,minimum size=5pt,inner sep=0pt] (10-) {};

\draw (1-) to (2-);
\draw (2-) to (3-);
\draw (3-) to (4-);
\draw (6-) to (7-);
\draw (7-) to (8-);
\draw (8-) to (10-);

%\node at (1,0) {$\alpha_1$};
%\node at (2,0) {$\alpha_2$};
%\node at (3,0) {$\alpha_3$};
%\node at (7,0) {$\alpha_{n-3}$};
%\node at (8,0) {$\alpha_{n-2}$};
%\node at (9.7,1.45) {$\alpha_{n-1}$};
%\node at (9.5,.05) {$\alpha_n$};
;\end{tikzpicture}\]
\caption{A type-D seaweed without seaweed shape}
\label{DSeaweedWithoutSeaweedShpe}
\end{figure}

While a seaweed without seaweed shape does not have block triangular form from which compositions can be obtained, type-D seaweeds \textit{with} seaweed shape share this property with seaweeds of all other classical types.  As in types B and C, there is a natural way to associate a partial composition of $n$ to each subset of simple roots defining a seaweed with seaweed shape.  What is different from the types B and C cases, however, is that this association is \textit{not} a bijection.  
Let $C_{\leq n}$ denote the set of strings of positive integers whose sum 
is less than or equal to $n$ and not equal to $n-1$, and (as before) call each integer in the string a \textit{part}.

\begin{remark}
In its natural representation, a seaweed with seaweed shape and $\alpha_{n-1} \not\in \Psi$ necessarily has $\alpha_n \not\in \Psi$; a seaweed $\mf{p}_n^\D(\ul{a} \dd \ul{b})$ with $\sum a_i = n-1$ and $\sum b_i \neq n-1$ is the same as the seaweed $\mf{p}_n^\D(\ul{a},1 \dd \ul{b})$.  We therefore exclude compositions of $n-1$ from our study.
\end{remark}

Let $\mathcal{P}(X)$ denote the power set of a set $X$.
Let $\Psi=\{\alpha_{s_1},\alpha_{s_2},\dots ,\alpha_{s_j}\}\subseteq \mathcal{P}(\Pi)$, and assume $s_1<s_2<\cdots <s_j$, where if $\alpha_n \in \Psi$, then $\alpha_{n-1} \in \Psi$, define $\varphi_\D:\mathcal{P}(\Pi)\rightarrow \Cn$ by

\[\varphi_{\D}(\Psi)=(s_1,s_2-s_1,s_3-s_2,\dots ,s_j-s_{j-1}),\]
and define

$$\mf{p}_n^\D(\Psi \dd \Psi')
=\mf{p}_n^\D(\ul{a} \dd \ul{b}),$$
\noindent 
where $\ul{a} = \varphi_\D(\Psi)$ and $\ul{b} = \varphi_\D(\Psi')$.  

\subsection{Type-D meanders}

In this section, we introduce type-D meanders with the goal of creating type-D analogues of Theorem \ref{DKformula} and Theorem \ref{symplectic index}.  A type-D meander is formed exactly as a type-C meander, and as in the type-C case, we will find it helpful to define a distinguished set of vertices called the \textit{tail} of the meander. However, the type-D tail can take several ``configurations".  The following critical definition sets the notation. 

\begin{definition}
Consider the seaweed $\mf{p}_n^\D(\ul{a} \dd \ul{b})$.  Assume $\sum a_i \geq \sum b_j$, and let $t = \sum a_i - \sum b_j$.  We define the type-$\D$ tail of $\mf{p}_n^\D(\ul{a} \dd \ul{b})$ to be
\begin{eqnarray}\label{typeDparts}
T_n^\D(\ul{a}\dd\ul{b})=
\begin{cases}
T_n^\C(\ul{a}\dd\ul{b}),
& \text{ if } t \text{ is even;} \\
T_n^\C(\ul{a}\dd\ul{b}) + v_{1+\sum a_i},
& \text{ if } t\text{ is odd and } \sum a_i < n; \\
T_n^\C(\ul{a}\dd\ul{b}) - v_n,
& \text{ if } t\text{ is odd and } \sum a_i = n. \\
\end{cases}
\end{eqnarray}

%The following Lemma verifies that the definition of the tail in Definition \ref{typeDparts}, is indeed well-defined.

%\begin{lemma}\label{TypeDTailLemma}
%Assume $n \geq \sum a_i \geq \sum b_j$.  If $t = \sum a_i - \sum b_j$ is even, then one of the following holds:

%\begin{enumerate}
%\item $n - \sum a_i$ is even and $n - \sum b_j$ is even,
%\item $n - \sum a_i$ is odd and $n - \sum b_j$ is odd. 
%\end{enumerate}

%If $t = \sum a_i - \sum b_j$ is odd, then one of the following holds:

%\begin{enumerate}
%\item $n - \sum a_i$ is even and $n - \sum b_j$ is odd,
%\item $n - \sum a_i$ is odd and $n - \sum b_j$ is even. 
%\end{enumerate}
%\end{lemma}

%\proof 
%If $t$ is even, then $2n - t - 2 \sum b_j = (n-\sum a_i)+(n-\sum b_j)$ is even.  The proof is similar for $t$ odd.  
%\qed 

We say that the tail, $T_n^\D(\ul{a}\dd\ul{b})$, has \textit{configuration} I, II, or III according to the three cases in (\ref{typeDparts}).  To ease notation, we will denote, for example, a seaweed
$\mf{p}_n^\D(\ul{a} \dd \ul{b})$ with \textit{tail configuration} III as $\mf{p}_n^\D\left((\ul{a} \dd \ul{b}), \textrm{III}\right)$, etc.
When the compositions $\ul{a}$ and $\ul{b}$ are explicit, we will find it convenient to use the alternative fractional notation $\mf{p}_n^\D \left( \dfrac{a_1 \dd ... \dd a_m}{b_1 \dd ... \dd b_r}, \textrm{III}\right)$.  
\end{definition}

\begin{example}
We illustrate the three cases in (\ref{typeDparts}).  The tail is indicated by yellow vertices.
\end{example}

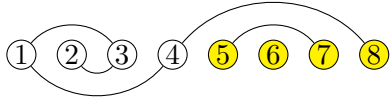
\begin{figure}[H]
\[\begin{tikzpicture}[scale=.67]

\vertex (1) at (1,0) {1};
\vertex (2) at (2,0) {2};
\vertex (3) at (3,0) {3};
\vertex (4) at (4,0) {4};
\vertex[fill=yellow] (5) at (5,0) {5};
\vertex[fill=yellow] (6) at (6,0) {6};
\vertex[fill=yellow] (7) at (7,0) {7};
\vertex[fill=yellow] (8) at (8,0) {8};

\draw (1) to [bend left=50] (3);
\draw (4) to [bend left=50] (8);
\draw (5) to [bend left=50] (7);
\draw (1) to [bend right=50] (4);
\draw (2) to [bend right=50] (3);

;\end{tikzpicture}\]

\caption{The meander for the seaweed $\mf{p}_{8}^\D \left( \dfrac{3 \dd 5}{4}, \textrm{I}\right)$}
\label{Tail1}
\end{figure}

\begin{figure}[H]
\[\begin{tikzpicture}[scale=.67]

\vertex (1) at (1,0) {1};
\vertex (2) at (2,0) {2};
\vertex (3) at (3,0) {3};
\vertex (4) at (4,0) {4};
\vertex (5) at (5,0) {5};
\vertex (6) at (6,0) {6};
\vertex[fill=yellow] (7) at (7,0) {7};
\vertex[fill=yellow] (8) at (8,0) {8};
\vertex (9) at (9,0) {9};

\draw (1) to [bend left=50] (7);
\draw (2) to [bend left=50] (6);
\draw (3) to [bend left=50] (5);
\draw (1) to [bend right=50] (3);
\draw (4) to [bend right=50] (6);

;\end{tikzpicture}\]
\caption{The meander for the seaweed $\mf{p}_{9}^\D \left( \dfrac{7}{3 \dd 3}, \textrm{II}\right)$} 
\label{Tail2}
\end{figure}
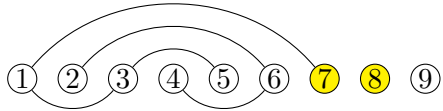

\begin{figure}[H]
\[\begin{tikzpicture}[scale=.67]

\vertex (1) at (1,0) {1};
\vertex (2) at (2,0) {2};
\vertex (3) at (3,0) {3};
\vertex (4) at (4,0) {4};
\vertex (5) at (5,0) {5};
\vertex (6) at (6,0) {6};
\vertex[fill=yellow] (7) at (7,0) {7};
\vertex[fill=yellow] (8) at (8,0) {8};
\vertex (9) at (9,0) {9};

\draw (1) to [bend left=50] (4);
\draw (2) to [bend left=50] (3);
\draw (5) to [bend left=50] (7);
\draw (8) to [bend left=50] (9);
\draw (1) to [bend right=50] (2);
\draw (3) to [bend right=50] (4);
\draw (5) to [bend right=50] (6);

;\end{tikzpicture}\]

\caption{The meander for the seaweed $\mf{p}_{9}^\D \left( \dfrac{4 \dd 3 \dd 2}{2 \dd 2 \dd 2}, \textrm{III}\right)$}
\label{Tail3}
\end{figure}
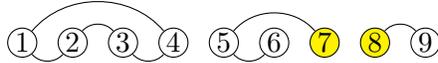

\subsection{Type-D formulas}

In this section, we establish a combinatorial formula for the index of a type-D seaweed analogous to Theorem \ref{DKformula} and Theorem \ref{symplectic index}.  We use this to classify Frobenius type-D seaweeds and extend these results to general index formulas.  

\subsubsection{Meander formula}

The following three theorems
give inductive formulas for the index.  These will be used to prove the combinatorial formula in Theorem \ref{typeD}.  

\begin{thm}[Panyushev and Yakimova \textbf{\cite{Panyushev1}}, Theorem 5.2]\label{Panyushev}
Let $\ul{a}\neq\emptyset$ and $\ul{b}\neq\emptyset$.
Consider the seaweed $\mf{p}_n^\D(\ul{a} \dd \ul{b})$, where 
$\ul{a}=(a_1,a_2,\dots ,a_m)$ and $\ul{b}=(b_1,b_2,\dots ,b_r)$.
\begin{enumerate}[\textup(i\textup)]
\item If $a_1=b_1$, then

\[\ind \mf{p}_n^\D(\ul{a} \dd \ul{b}) 
=a_1+\ind \mf{p}_{n-a_1}^\D((a_2,a_3,\dots a_m) \dd (b_2,b_3,\dots b_r)).\]
\item If $a_1<b_1$, then

\[\ind \mf{p}_n^\D(\ul{a} \dd \ul{b})=
\begin{cases}
\ind \mf{p}_{n-a_1}^\D((a_2,a_3,\dots a_m) \dd (b_1-2a_1,a_1,b_2,b_3,\dots b_r)),
& \text{ if }a_1\leq b_1/2; \\
\ind \mf{p}_{n-b_1+a_1}^\D((2a_1-b_1,a_2,a_3,\dots a_m) \dd (a_1,b_2,b_3,\dots b_r)),
& \text{ if }a_1> b_1/2.
\end{cases}\]
\end{enumerate}
\end{thm}

Note that if $a_1>b_1$, we can use the fact that 
$\mf{p}_n^\D(\ul{a} \dd \ul{b})\cong\mf{p}_n^\D(\ul{b} \dd \ul{a})$.

\begin{thm}[Dvorsky \textbf{\cite{Dvorsky}}, Theorem 4.1]\label{Dvorsky1}
Let $\ul{a}=(a_1,\dots ,a_m)$ with $\sum a_i = n$.  For parabolic subalgebras of $\mf{so}(2n)$, 
\begin{enumerate}[\textup(i\textup)]
   \item if $n$ is even, then $\ind\mf{p}_n^\D(\ul{a} \dd \emptyset) = \displaystyle\sum_{i=1}^m \lf\frac{a_i}{2}\rf$, 
   \item if $n$ is odd, then 
   
\[\ind \mf{p}_n^\D(\ul{a} \dd \emptyset)=
\begin{cases}
\displaystyle\sum_{i=1}^m \lf\frac{a_i}{2}\rf + 1, 
& \text{ if }a_m = 1; \\
\displaystyle\sum_{i=1}^m \lf\frac{a_i}{2}\rf - 1, 
& \text{ otherwise.}
\end{cases}\]
\end{enumerate}
\end{thm}

\begin{thm}[Dvorsky \textbf{\cite{Dvorsky}}, Theorem 4.3]\label{Dvorsky2}

Let $\ul{a}=(a_1,\dots ,a_m)$ with $\sum a_i = a$.  Let $k = n-a \geq 2$.  
\begin{enumerate}[\textup(i\textup)]
   \item If $a$ is even, then $\ind\mf{p}_n^\D(\ul{a} \dd \emptyset) = k + \displaystyle \sum_{i=1}^m \lf\frac{a_i}{2}\rf$.  
   \item If $a$ is odd, then $\ind\mf{p}_n^\D(\ul{a} \dd \emptyset) = k-1 + \displaystyle \sum_{i=1}^m \lf\frac{a_i}{2}\rf$.
\end{enumerate}
\end{thm}

We have the following corollary of the above theorems.  This is the type-D analogue of a type-C result used to prove the type-C combinatorial formula in Coll et al. \textbf{\cite{Coll4}}. 

\begin{cor}\label{Dcorollary}
Consider the seaweed $\mf{p}_{n+k}^\D(\ul{a} \dd \ul{b})$, where 
$\ul{a}=(a_1,a_2,\dots ,a_m)$ and $\ul{b}=(b_1,b_2,\dots ,b_r)$.
Suppose $\displaystyle n+k>n=\sum_{i=1}^{m}a_i\geq \sum_{i=1}^{r}b_i$.  Let $\displaystyle t = \sum_{i=1}^{m}a_i - \sum_{i=1}^{r}b_r$.  

\begin{enumerate}[\textup(i\textup)]
   \item If $t$ is even, then $\ind \mf{p}_{n+k}^\D(\ul{a} \dd \ul{b})=k+\ind \mf{p}_n^\C(\ul{a} \dd \ul{b})$. 
   \item If $t$ is odd, then $\ind \mf{p}_{n+k}^\D(\ul{a} \dd \ul{b})=k-1+\ind \mf{p}_n^\C(\ul{a} \dd \ul{b})$. 
\end{enumerate}
\end{cor}

Given a type-D meander $M_n^\D(\ul{a} \dd \ul{b})$,
we define top and bottom bijections $t$ and $b$ as before,
and we associate to the meander the permutation $\sigma_{n,\ul{a},\ul{b}}$ 
defined by $\sigma_{n,\ul{a},\ul{b}}(j)=t(b(j))$. For example, if $\ul{a}=(4,3,2)$ and $\ul{b}=(2,2,2)$ 
are strings in $C_{\leq 9}$,
then the associated permutation written in disjoint cycle form
is $\sigma_{9,\ul{a},\ul{b}}=(1,3)(2,4)(5,7,6)(8,9)$.  We can now establish the combinatorial index formula.

\begin{thm}\label{typeD}
Consider the seaweed $\mf{p}_n^\D(\ul{a} \dd \ul{b})$.  Let $T = T_n^\D(\ul{a}\dd\ul{b})$.  

\begin{enumerate}[\textup(i\textup)]
\item The index of $\mf{p}_n^\D(\ul{a} \dd \ul{b})$  is equal to $2C + \widetilde{P}$, where $C$ is the number of cycles in $M_n^\D(\ul{a} \dd \ul{b})$, and $\tilde{P}$ is the number of paths containing either zero or two vertices from $T$ in $M_n^\D(\ul{a} \dd \ul{b})$.

\item  The index of $\mf{p}_n^\D(\ul{a} \dd \ul{b})$ is equal to the number of cycles containing either
zero or two integers from $T$ in the disjoint cycle decomposition of $\sigma_{n,\ul{a},\ul{b}}$ 
\textup(here we view $T$ as a set of integers\textup).

\end{enumerate}
\end{thm}

\proof
A cycle in $M_n^\D(\ul{a} \dd \ul{b})$ cannot contain vertices from $T$, and breaks into two cycles 
in $\sigma_{n,\ul{a},\ul{b}}$. A path in $M_n^\D(\ul{a} \dd \ul{b})$ will be a cycle in $\sigma_{n,\ul{a},\ul{b}}$
containing the labels of all the vertices in the path. Thus, ($i$) and ($ii$) are equivalent, so it suffices to prove ($i$).
By Corollary \ref{Dcorollary} and symmetry, it suffices to consider the case when $\sum a_i =n$ and $\sum b_i \leq n$.
Now, induct on $n$. The base case is trivial.  Given a meander $G$, let $f(G)$ denote the number of cycles plus the number of connected
components containing either zero or two vertices from $T$ in $G$.

For the inductive step, first consider the case where $\ul{b}=\emptyset$.  If $n$ is even, then all vertices belong to the tail $T$.  If $n$ is odd, then all vertices except $v_n$ belong to $T$.  
There are no cycles in $M_n^\D(\ul{a} \dd \emptyset)$.  If $n$ is even, then there are no connected components containing zero 
vertices from $T$.  If $n$ is odd and $a_m = 1$, then there is one connected component containing zero vertices from $T$.  If $n$ is odd and $a_m > 1$, then there are no connected components containing zero vertices from $T$, and there is one connected component containing one vertex from $T$. Since each block of vertices $V_i$ is assigned $\lf a_i/2\rf$ top edges, 
it follows from Theorem~\ref{Dvorsky1} that

\[f\left(M_n^\D(\ul{a} \dd \emptyset)\right)=
\begin{cases}
\displaystyle\sum_{i=1}^{m}\lf\frac{a_i}{2}\rf = \ind \mf{p}_n^\D(\ul{a} \dd \emptyset),
& \text{ if } n \text{ is even;} \\
\displaystyle\sum_{i=1}^{m}\lf\frac{a_i}{2}\rf + 1 = \ind \mf{p}_n^\D(\ul{a} \dd \emptyset),
& \text{ if } n \text{ is odd and } a_m = 1; \\
\displaystyle\sum_{i=1}^{m}\lf\frac{a_i}{2}\rf - 1 = \ind \mf{p}_n^\D(\ul{a} \dd \emptyset),
& \text{ if } n \text{ is odd and } a_m > 1. \\
\end{cases}\]

To complete the inductive step, consider the case where $\ul{b}\neq\emptyset$. Suppose $a_1=b_1$.
Let $H$ denote the subgraph of $M_n^\D(\ul{a} \dd \ul{b})$ induced by the vertices labeled 1 through $a_1$,
and let $G$ denote the subgraph induced by the remaining vertices. Then $M_n^\D(\ul{a} \dd \ul{b})=H+G$
and $H$ contains no vertices from $T$. Clearly $f(H)=a_1$, and using the inductive hypothesis on $G$ we have
\begin{eqnarray}\label{10}
f\left(M_n^\D(\ul{a} \dd \ul{b})\right)=f(H)+f(G)
=a_1+\ind \mf{p}_{n-a_1}^\D((a_2,a_3,\dots a_m) \dd (b_2,b_3,\dots b_t)).
\end{eqnarray}
By Theorem \ref{Panyushev}, the right-hand side of Equation (\ref{10}) is equal to $\ind \mf{p}_n^\D(\ul{a} \dd \ul{b})$.

Suppose $a_1\leq b_1/2$. By Theorem \ref{WindingDown}, the meander

$$
G = M_{n-a_1}^\D((a_2,a_3,\dots a_m) \dd (b_1-2a_1,a_1,b_2,b_3,\dots b_t))
$$

\noindent
can be obtained from $M_n^\D(\ul{a} \dd \ul{b})$ by edge contractions that
do not delete vertices from $T$. Thus, by induction, we have
\begin{eqnarray}\label{11}
f\left(M_n^\D(\ul{a} \dd \ul{b})\right)=f(G)
=\ind \mf{p}_{n-a_1}^\D((a_2,a_3,\dots a_m) \dd (b_1-2a_1,a_1,b_2,b_3,\dots b_t)).
\end{eqnarray}
And by Theorem \ref{Panyushev}, the right-hand side of Equation (\ref{11}) is equal to $\ind \mf{p}_n^\D(\ul{a} \dd \ul{b})$.

Similarly, suppose $a_1>b_1/2$. By Theorem \ref{WindingDown},
the meander 

$$ G=M_{n-b_1+a_1}^\D((2a_1-b_1,a_2,a_3,\dots a_m) \dd (a_1,b_2,b_3,\dots b_t))
$$
can be obtained from $M_n^\D(\ul{a} \dd \ul{b})$ by edge contractions that
do not delete vertices from $T$. Again, by induction, we have
\begin{eqnarray}\label{12}
f\left(M_n^\D(\ul{a} \dd \ul{b})\right)=f(G)
=\ind \mf{p}_{n-b_1+a_1}^\D((2a_1-b_1,a_2,a_3,\dots a_m) \dd (a_1,b_2,b_3,\dots b_t)).
\end{eqnarray}
Again, by Theorem \ref{Panyushev}, the right-hand side of Equation (\ref{12}) is equal to $\ind \mf{p}_n^\D(\ul{a} \dd \ul{b})$.
\qed 

\begin{example}
The seaweeds whose meanders are given by Figures \ref{Tail1}, \ref{Tail2}, and \ref{Tail3} have index one, two, and two, respectively.  
\end{example}

With the established definition of the type-D tail, we can now concisely state a type-D analogue to the visual for type-C Frobenius seaweeds.  

\begin{thm}\label{DForest}
A type-D seaweed is Frobenius if and only if its corresponding meander graph is a forest rooted in the tail.  
\end{thm}

The following corollary reduces index computation for all type-D seaweeds with a tail of configuration I to previously-solved cases.  

\begin{thm}\label{configI}
The index of $\mf{p}_n^\D\left((\ul{a} \dd \ul{b}), \I \right)$ equals the index of $\mf{p}_n^\C (\ul{a} \dd \ul{b})$.  
\end{thm}

%\proof
%This follows from the definition of the type-D tail and Theorem~\ref{typeD}.  
%\qed 

\begin{remark}
As a corollary of Theorem \ref{configI}, Theorem \ref{3 parts thm1} holds when $c=n-2$, and case \textup(ii\textup) in each of Theorems \ref{3 parts thm2} and \ref{3 parts thm4} hold in type D.  
\end{remark}

%\bigskip
We also have the following obstruction theorem to a type-D seaweed with tail of configuration $\II$ being Frobenius.  In such seaweeds, the vertex $v_n$ is always a component separated from the tail.  

\begin{thm}\label{configII}
A type-D seaweed with a tail of configuration $\II$ is never Frobenius.  
\end{thm}

\subsubsection{Greatest common divisor formulas}

In this subsection, we find explicit greatest common divisor formulas for the index in terms of elementary functions of the parts that determine the seaweed.
%We start by classifying type-D Frobenius seaweeds.  
We consider seaweeds where $\ul{a}$ and $\ul{b}$ have a small number of parts.
Theorem \ref{Dvorsky1} directly covers all cases when either $\ul{a}=\emptyset$ or $\ul{b}=\emptyset$. The
next case we consider is when $\ul{a}$ and $\ul{b}$ each have one part. This case is easily handled
by applying Theorem \ref{Panyushev} of Panyushev and Theorem \ref{Dvorsky1} of Dvorsky and should be considered
a corollary of these results.

\begin{thm}\label{two block}
If $a=b$, then $\ind \mf{p}_n^\D\frac{~a~}{~b~}=n$.
Otherwise,

\[\ind \mf{p}_n^\D\frac{~a~}{~b~}=
\begin{cases}
n-a + \lf\frac{a-b}{2}\rf, & \text{ if }a\text{ is even};\\
n-a-1+\lf\frac{a-b-1}{2}\rf, & \text{ if }a\text{ is odd}.
\end{cases}\]

\end{thm}

%\begin{proof}

%If $n>a$ then $\ind \mf{p}_n^\C((a) \dd (b))=n-a+\ind \mf{p}_{n-a}^\C((a) \dd (b))$, so it suffices to show that
%\begin{equation}\label{2 parts eq}
%\ind \mf{p}_n^\C((n) \dd (b))=
%\begin{cases}
%\lf\frac{n-b}{2}\rf & \text{ if }n\text{ is even}\\
%\lf\frac{n-b-1}{2}\rf & \text{ if }n\text{ is odd}.
%\end{cases}
%\end{equation}

%Suppose $b\leq n/2$. By Theorem \ref{C inductive} and Theorem \ref{C parabolic} we have
%\[\ind \mf{p}_n^\C((n) \dd (b))=\ind \mf{p}_n^\C((b) \dd (n))=\ind \mf{p}_{n-b}^\C((\emptyset) \dd (n-2b,b))
%=\lf\frac{n-2b}{2}\rf+\lf\frac{b}{2}\rf.\]
%If $n$ is even then $n-2b$ is even and
%\[\lf\frac{n-2b}{2}\rf+\lf\frac{b}{2}\rf=\frac{n-2b}{2}+\lf\frac{b}{2}\rf=\lf\frac{n-b}{2}\rf.\]
%If $n$ is odd then $n-2b$ is odd and
%\[\lf\frac{n-2b}{2}\rf+\lf\frac{b}{2}\rf=\frac{n-2b-1}{2}+\lf\frac{b}{2}\rf=\lf\frac{n-b-1}{2}\rf.\]

%Suppose $b>n/2$. We prove \eqref{2 parts eq} by induction on $n$. The base case is trivial. For the inductive step, 
%use Theorem \ref{C inductive}:
%\[\ind \mf{p}_n^\C((n) \dd (b))=\ind \mf{p}_n^\C((b) \dd (n))=\ind \mf{p}_{b}^\C((2b-n) \dd (b)).\]
%Using the inductive hypothesis, it is easy to show that \eqref{2 parts eq} holds by considering all four cases 
%for the parity of $n$ and $b$.

%\end{proof}

The next case we consider is when $\ul{a}$ and $\ul{b}$ have a total of three parts.  
Having dispensed with configurations I and II in Theorems \ref{configI} and \ref{configII}, respectively,  we need only consider seaweeds of the form $\mf{p}_n^\D\left(\dfrac{a \dd b}{c}, \textrm{III}\right)$.  In such seaweeds, the component containing $v_n$ can contribute to the index differently depending on how $b$ and $n-c$ are related, as the following figures illustrate.  
The tail is indicated by yellow vertices. 

\begin{figure}[H]
\[\begin{tikzpicture}[scale=.67]

\vertex (1) at (1,0) {1};
\vertex (2) at (2,0) {2};
\vertex (3) at (3,0) {3};
\vertex (4) at (4,0) {4};
\vertex (5) at (5,0) {5};
\vertex[fill=yellow] (6) at (6,0) {6};
\vertex[fill=yellow] (7) at (7,0) {7};
\vertex (8) at (8,0) {8};

\draw (1) to [bend left=50] (5);
\draw (2) to [bend left=50] (4);
\draw (6) to [bend left=50] (8);
\draw (1) to [bend right=50] (5);
\draw (2) to [bend right=50] (4);

;\end{tikzpicture}\]
\caption{The meander for the seaweed $\mf{p}_{8}^\D \left( \dfrac{5 \dd 3}{5}, \III\right)$ has $b = n-c$.}
\label{Case1} 
\end{figure}
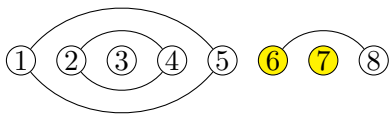

\begin{figure}[H] 
\[\begin{tikzpicture}[scale=.67]

\vertex (1) at (1,0) {1};
\vertex (2) at (2,0) {2};
\vertex (3) at (3,0) {3};
\vertex (4) at (4,0) {4};
\vertex (5) at (5,0) {5};
\vertex[fill=yellow] (6) at (6,0) {6};
\vertex[fill=yellow] (7) at (7,0) {7};
\vertex[fill=yellow] (8) at (8,0) {8};
\vertex[fill=yellow] (9) at (9,0) {9};
\vertex (10) at (10,0) {10};

\draw (1) to [bend left=50] (7);
\draw (2) to [bend left=50] (6);
\draw (3) to [bend left=50] (5);
\draw (8) to [bend left=50] (10);
\draw (1) to [bend right=50] (5);
\draw (2) to [bend right=50] (4);

;\end{tikzpicture}\]

\caption{The meander for the seaweed $\mf{p}_{10}^\D \left( \dfrac{7 \dd 3}{5}, \textrm{III}\right)$ has $b < n-c$.}
\label{Case2} 
\end{figure}

%\begin{figure}[H]

%\[\begin{tikzpicture}
%[decoration={markings,mark=at position 0.6 with 
%{\arrow{angle 90}{>}}}]

%\vertex (1) at (1,0) {1};
%\vertex (2) at (2,0) {2};
%\vertex (3) at (3,0) {3};
%\vertex (4) at (4,0) {4};
%\vertex (5)[fill=yellow] at (5,0) {5};
%\vertex (6)[fill=yellow] at (6,0) {6};
%\vertex[fill=yellow] (7) at (7,0) {7};
%\vertex[fill=yellow] (8) at (8,0) {8};
%\vertex (9) at (9,0) {9};

%\draw (1) to [bend left=50] (6);
%\draw (2) to [bend left=50] (5);
%\draw (3) to [bend left=50] (4);
%\draw (7) to [bend left=50] (9);
%\draw (1) to [bend right=50] (4);
%\draw (2) to [bend right=50] (3);

%;\end{tikzpicture}\]

%\caption{The meander for the seaweed $\mf{p}_{9}^\D \left( %\dfrac{6 \dd 3}{4}, \III\right)$ has $b < n-c$.} 

%\end{figure}

\begin{figure}[H]
\[\begin{tikzpicture}[scale=.67]

\vertex (1) at (1,0) {1};
\vertex (2) at (2,0) {2};
\vertex (3) at (3,0) {3};
\vertex (4) at (4,0) {4};
\vertex (5) at (5,0) {5};
\vertex (6) at (6,0) {6};
\vertex (7) at (7,0) {7};
\vertex[fill=yellow] (8) at (8,0) {8};
\vertex[fill=yellow] (9) at (9,0) {9};
\vertex (10) at (10,0) {10};

\draw (1) to [bend left=50] (4);
\draw (2) to [bend left=50] (3);
\draw (5) to [bend left=50] (10);
\draw (6) to [bend left=50] (9);
\draw (7) to [bend left=50] (8);
\draw (1) to [bend right=50] (7);
\draw (2) to [bend right=50] (6);
\draw (3) to [bend right=50] (5);

;\end{tikzpicture}\]

\caption{The meander for the seaweed $\mf{p}_{10}^\D \left( \dfrac{4 \dd 6}{7}, \textrm{III}\right)$ has $b > n-c$.}
\label{Case3} 
\end{figure}

\subsubsection{Seaweeds $\mf{p}_n^\D\left(\dfrac{a \dd b}{c}, \textrm{III}\right)$ }

The analysis of these seaweeds breaks into three cases, illustrated by the examples in Figures \ref{Case1}, \ref{Case2}, and \ref{Case3}, respectively.

\bigskip
\noindent
\textbf{Case 1: } $b = n-c$ 

If $b = n-c$, then $\mf{p}_n^\D\left(\dfrac{a \dd b}{c}, \textrm{III}\right)$ cannot be Frobenius since the components on the first $a$ vertices are always separated from the tail.  Moreover, we have the following more general index formula.  

\begin{thm}
If $b = n-c$, then 

\[\ind \mf{p}_n^\D\left(\dfrac{a \dd b}{c}, \III\right)=
\begin{cases}
a, & \text{ if }b = 1;\\
a + \lf \displaystyle \frac{b-3}{2} \rf, & \text{ if }b \geq 3.
\end{cases}\]
\end{thm}

%\proof  This is a corollary of Theorems \ref{Panyushev} and \ref{Dvorsky1}. \qed

\noindent
\textbf{Case 2: } $b < n-c$

The seaweed in Figure \ref{Case2} is Frobenius by Theorem~\ref{DForest}. Note that the subgraph on vertices $v_1$ through $v_7$ yields a Frobenius type-C meander.  In general, when $v_a$ and $v_{a+1}$ are tail vertices, the meander can be separated into two parts: one on the first $a$ vertices, and the other on the last $b$ vertices, with no arc connecting the two subgraphs. We find that for such a seaweed to be Frobenius, it must have $b=2$ or $3$, a result which holds for general seaweeds of this form.  

\begin{thm}\label{Attaching Lemma}
If $\mf{p}_n^\D\left( \dfrac{a_1 \dd ... \dd a_m}{b_1 \dd ... \dd b_r}, \III\right)$ is Frobenius and $a_m < n - \sum b_i$, then $a_m = 2$ or $3$.  Furthermore, $\mf{p}_{n-a_m}^\C \dfrac{a_1 \dd ... \dd a_{m-1}}{b_1 \dd ... \dd b_r}$ is Frobenius.  
\end{thm}

\proof
If $a_m = 1$, then $v_n$ is a path separated from the tail and contributes $1$ to the index.  If $a_m \geq 4$, then the component containing $v_{n-1}$ is a path with two ends in the tail, which contributes one to the index. 
\qed 

%As a corollary, we classify Frobenius parabolic subalgebras of $\mf{so}(2n)$. 

%\begin{thm}
%Frobenius parabolic subalgebras of $\mf{so}(2n)$ have one of the following block decompositions:

%\begin{itemize}
%\item $\mf{p}_n^\D\left( \dfrac{1 \dd ... \dd 1 \dd 2}{\emptyset}, \III\right)$ with $n$ odd,
%\item $\mf{p}_n^\D\left( \dfrac{1 \dd ... \dd 1 \dd 3}{\emptyset}, \III\right)$ with $n$ odd,
%\item $\mf{p}_n^\D\left( \dfrac{1 \dd ... \dd 1}{\emptyset}, \II\right)$ with $n$ even.
%\end{itemize}
%\end{thm}

As a corollary of 
Theorems \ref{two block} and \ref{Attaching Lemma}, we can classify which seaweeds in Case 2 are Frobenius.  

\begin{thm}\label{TypeDThreeBlock}
The seaweed $\mf{p}_n^\D\left(\dfrac{a \dd b}{c}, \III\right)$ with $b < n - c$ is Frobenius if and only if one of the following holds:
\begin{enumerate}[\textup(i\textup)]
\item $b=2$ and $c=n-3$,
\item $b=3$, $c=n-5$, and $n$ is odd.
\end{enumerate}
\end{thm}

The following theorem gives a relation between type-C Frobenius seaweeds and type-D Frobenius seaweeds in Case 2.  It serves as a partial converse to Theorem \ref{Attaching Lemma}.  

\begin{thm}\label{CtoD}
Let $\mf{p}_{n}^\C \dfrac{a_1 \dd ... \dd a_m}{b_1 \dd ... \dd b_r}$ be Frobenius with $\sum a_i > \sum b_j$.  

\begin{enumerate}[\textup(i\textup)]
\item If $\sum a_i - \sum b_j$ is odd, then
$\mf{p}_{2 + n}^\D\left( \dfrac{a_1 \dd ... \dd a_m \dd 2}{b_1 \dd ... \dd b_r}, \III\right)$ is Frobenius.
\item If $\sum a_i - \sum b_j$ is even, then
$\mf{p}_{3 + n}^\D\left( \dfrac{a_1 \dd ... \dd a_m \dd 3}{b_1 \dd ... \dd b_r}, \III\right)$ is Frobenius.  
\end{enumerate}
\end{thm}

\begin{example}
The seaweed $\mf{p}_{7}^\C \frac{4 \dd 3}{6}$ is Frobenius by Theorem~\ref{symplectic index}.   We can apply Theorem \ref{CtoD} to conclude $\mf{p}_{9}^\D \frac{4 \dd 3 \dd 2}{6}$ in Figure \ref{DFrobeniusAttaching} is Frobenius.  
\end{example}
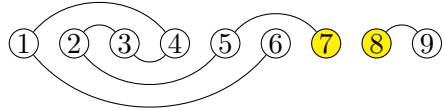
\begin{figure}[H]
\[\begin{tikzpicture}[scale=.67]

\vertex (1) at (1,0) {1};
\vertex (2) at (2,0) {2};
\vertex (3) at (3,0) {3};
\vertex (4) at (4,0) {4};
\vertex (5) at (5,0) {5};
\vertex (6) at (6,0) {6};
\vertex[fill=yellow] (7) at (7,0) {7};
\vertex[fill=yellow] (8) at (8,0) {8};
\vertex (9) at (9,0) {9};

\draw (1) to [bend left=50] (4);
\draw (2) to [bend left=50] (3);
\draw (5) to [bend left=50] (7);
\draw (8) to [bend left=50] (9);
\draw (1) to [bend right=50] (6);
\draw (2) to [bend right=50] (5);
\draw (3) to [bend right=50] (4);

;\end{tikzpicture}\]

\caption{The seaweed $\mf{p}_{9}^\D \left( \dfrac{4 \dd 3 \dd 2}{6}, \textrm{III}\right)$ is Frobenius.}
\label{DFrobeniusAttaching}
\end{figure}

This example, together with Theorem~\ref{Attaching Lemma}, gives some insight into Frobenius seaweeds of the form $\mf{p}_n^\D\left(\dfrac{a|b|k}{c}, \textrm{III}\right)$ for specific $k$ with $k < n -c$.

\begin{thm}\label{TypeDFourBlock}
The seaweed $\mf{p}_n^\D\left(\dfrac{a|b|k}{c}, \III\right)$ with $k=2$ or $3$ is Frobenius if and only if one of the following holds:
\begin{enumerate}[\textup(i\textup)]
\item $k=2$, $c=n-3$, and $\gcd(a+b,b+c)=1$,
\item $k=3$, $c=n-5$, and $\gcd(a+b,b+c)=1$,
\item $k=2$, $c=n-5$, and $\gcd(a+b,b+c)=2$ with $a, b$, and $c$ all odd.
\end{enumerate}
\end{thm} 

\proof
This is immediate from Theorems \ref{3 parts thm2} and \ref{Attaching Lemma}.  
\qed 

%We expect further specific four-block cases for which a $\gcd$ condition on three blocks determines whether a seaweed is Frobenius to not exist based on the work of Coll et al. in \textbf{\cite{Coll4}}.  

\bigskip
\noindent
\textbf{Case 3: } $b > n-c$

\noindent 
If $b > n-c$, then the tail, $T$, of a 
Frobenius seaweed $\mf{p}_n^\D\left(\dfrac{a \dd b}{c}, \textrm{III}\right)$ must have limited size.

\begin{thm} If $\mf{p}_n^\D\left(\dfrac{a \dd b}{c}, \III\right)$ is Frobenius, then  
$c=n-3$ or $c=n-5$.  In particular, $|T| = 2$ or $4$.  
\end{thm}

\proof
Suppose, for a contradiction, that $c \geq n-7$.  Then $\mf{p}_{n-1}^\C\dfrac{a \dd 1 \dd b-2}{c}$ is Frobenius by Theorem \ref{symplectic index}.  By Theorem~\ref{c necessary}, the seaweed $\mf{p}_{n-1}^\C\dfrac{a \dd 1 \dd b-2}{c}$ must have exactly $n-1-c \geq 6$ odd integers among its parts -- a contradiction. When $c=n-3$, the tail of $\mf{p}_n^\D\left(\dfrac{a \dd b}{c}, \III\right)$ is given by $T=\{v_{n-2}, v_{n-1}\}$.  When $c=n-5$, the tail of $\mf{p}_n^\D\left(\dfrac{a \dd b}{c}, \III\right)$
is given by $T=\{ v_{n-4},v_{n-3}, v_{n-2}, v_{n-1} \}$.
\qed 

Remaining in Case 3, we now examine different tail sizes.

\bigskip
\noindent
\textbf{Case 3.1:  $|T|=2$}

We say the seaweed $\mf{p}_n^\D \dfrac{a_1 \dd ... \dd a_m}{b_1 \dd ... \dd b_r}$ has \textit{type-A homotopy type} $H(k)$ if
$\mf{p}_n^\A \dfrac{a_1 \dd ... \dd a_m \dd n - \sum a_i}{b_1 \dd ... \dd b_r \dd n - \sum b_j}$ has homotopy type $H(k)$. To find a type-D seaweed's type-A homotopy type, we simply add an additional part (if necessary) to each partial composition and consider 
%the seaweed formed by 
two full compositions of $n$.

\begin{thm}\label{tail two homotopy type}
If $\mf{p}_n^\D\left(\dfrac{a \dd b}{c}, \III\right)$ is Frobenius, then it has type-A homotopy type $H(3)$ or $H(1)$.  
\end{thm} 

\proof
Such a seaweed consists of two paths: one containing $v_{n-2}$ and one containing $v_{n-1}$.  Exactly one contains $v_n$.  If $v_{n-2}$ and $v_n$ are on the same path, then $\mf{p}_n^\A\dfrac{a \dd b}{c \dd 3}$ has homotopy type $H(3)$.  If $v_{n-1}$ and $v_n$ are on the same path, then $\mf{p}_n^\A\dfrac{a \dd b}{c \dd 3}$ has homotopy type $H(1)$.  
\qed

\begin{example}
The following figures illustrate the two type-A homotopy types described in Theorem \ref{tail two homotopy type}.  In each figure, the meander without the dotted lower arc is a type-D meander;  with the dotted lower arc included, it is the type-A meander from which the type-A homotopy type is discerned.
\end{example}
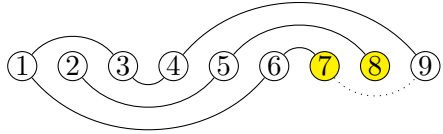
\begin{figure}[H]
\[\begin{tikzpicture}[scale=.67]

\vertex (1) at (1,0) {1};
\vertex (2) at (2,0) {2};
\vertex (3) at (3,0) {3};
\vertex (4) at (4,0) {4};
\vertex (5) at (5,0) {5};
\vertex(6) at (6,0) {6};
\vertex[fill=yellow] (7) at (7,0) {7};
\vertex[fill=yellow]  (8) at (8,0) {8};
\vertex (9) at (9,0) {9};

\draw (1) to [bend left=50] (3);
\draw (4) to [bend left=50] (9);
\draw (5) to [bend left=50] (8);
\draw (6) to [bend left=50] (7);
\draw (1) to [bend right=50] (6);
\draw (2) to [bend right=50] (5);
\draw (3) to [bend right=50] (4);
\draw[dotted] (7) to [bend right=50] (9);

;\end{tikzpicture}\]
\caption{The seaweed $\mf{p}_{9}^\D \left( \dfrac{3 \dd 6}{6}, \III\right)$ has type-A homotopy type $H(3)$.} 
\label{DFrobeniusH3}
\end{figure}

\begin{figure}[H]
\[\begin{tikzpicture}[scale=.67]

\vertex (1) at (1,0) {1};
\vertex (2) at (2,0) {2};
\vertex (3) at (3,0) {3};
\vertex (4) at (4,0) {4};
\vertex (5) at (5,0) {5};
\vertex (6) at (6,0) {6};
\vertex (7) at (7,0) {7};
\vertex[fill=yellow] (8) at (8,0) {8};
\vertex[fill=yellow] (9) at (9,0) {9};
\vertex (10) at (10,0) {10};

\draw (1) to [bend left=50] (4);
\draw (2) to [bend left=50] (3);
\draw (5) to [bend left=50] (10);
\draw (6) to [bend left=50] (9);
\draw (7) to [bend left=50] (8);
\draw (1) to [bend right=50] (7);
\draw (2) to [bend right=50] (6);
\draw (3) to [bend right=50] (5);
\draw[dotted] (8) to [bend right=50] (10);

;\end{tikzpicture}\]

\caption{The seaweed $\mf{p}_{10}^\D \left( \dfrac{4 \dd 6}{7}, \textrm{III}\right)$ has type-A homotopy type $H(1)$.}
\label{DFrobeniusH1}
\end{figure}

Letting $d = k$ in Theorem 5.2 of \textbf{\cite{Coll3}} gives the following useful corollary.   

\begin{thm}\label{HomotopyTypeIFF}
The seaweed $\mf{p} = \mf{p}_n^\A\dfrac{a \dd b}{c \dd k}$ has homotopy type $H(k)$ if and only if 

$$\gcd(a+b,b+c) = k.$$ 
\noindent
Futhermore, $a$, $b$ and $c$ are all multiples of $k$.  If $\hat{a} = \frac{a}{k}$, $\hat{b} = \frac{b}{k}$, and $\hat{c} = \frac{c}{k}$, then $M_n^\A\dfrac{\hat{a} \dd \hat{b}}{\hat{c} \dd 1}$ consists of a single path.  
\end{thm}

We have the following theorem as a corollary.

\begin{thm}\label{HomotopyTypeH3}
If $\mf{p}_n^\D\left(\dfrac{a \dd b}{c}, \III\right)$ has type-A homotopy type $H(3)$, then it is Frobenius if and only if $\gcd(a+b,b+c)=3$.  
\end{thm}

If $\mf{p}_n^\A\dfrac{a \dd b}{c \dd 3}$ has homotopy type $H(1)$, then $\gcd(a+b,b+c)=1$.  However, this is not enough to guarantee $\mf{p}_n^\D\left(\dfrac{a \dd b}{c}, \III\right)$ is Frobenius as the following example illustrates.  

\begin{example}
The seaweed $\mf{p}_{10}^\D \left( \dfrac{6 \dd 4}{7}, \textrm{III}\right)$ has $\gcd(6+4,4+7)=1$ but is not Frobenius.  However, $\mf{p}_{10}^\A \dfrac{6 \dd 4}{7 \dd 3}$ does indeed have homotopy type $H(1)$.  See Figure \ref{NotFrobeniusH1}.  
\end{example}
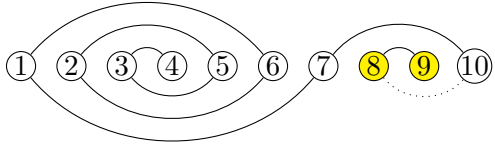
\begin{figure}[H]
\[\begin{tikzpicture}[scale=.67]

\vertex (1) at (1,0) {1};
\vertex (2) at (2,0) {2};
\vertex (3) at (3,0) {3};
\vertex (4) at (4,0) {4};
\vertex (5) at (5,0) {5};
\vertex (6) at (6,0) {6};
\vertex (7) at (7,0) {7};
\vertex[fill=yellow] (8) at (8,0) {8};
\vertex[fill=yellow] (9) at (9,0) {9};
\vertex (10) at (10,0) {10};

\draw (1) to [bend left=50] (6);
\draw (2) to [bend left=50] (5);
\draw (3) to [bend left=50] (4);
\draw (7) to [bend left=50] (10);
\draw (8) to [bend left=50] (9);
\draw (1) to [bend right=50] (7);
\draw (2) to [bend right=50] (6);
\draw (3) to [bend right=50] (5);
\draw (8)[dotted] to [bend right=50] (10);
;\end{tikzpicture}\]
\caption{The seaweed $\mf{p}_{10}^\D \left( \dfrac{6 \dd 4}{7}, \textrm{III}\right)$ has type-A homotopy type $H(1)$.}
\label{NotFrobeniusH1}
\end{figure}

Comparing this example to the seaweed in Figure \ref{DFrobeniusH1}, we notice that for such $\mf{p}_{n}^\D \left( \dfrac{a \dd b}{c}, \textrm{III}\right)$ to be Frobenius, we need a condition that guarantees vertices $v_{n-2}$ and $v_{n-1}$ are on different components.  
To find this condition, we start with the following theorem.

\begin{thm}\label{meander deltas}

Let $\underline{a}=(a_1,a_2,\dots a_k)$ and $\underline{b}=(n)$ be compositions of $n$. 
Let $\sigma=\sigma_{\underline{a},\underline{b}}$. Consider the sequence
$d_{\underline{a},\underline{b}}=(\sigma(1)-1,\sigma(2)-2,\dots ,\sigma(n)-n)$. Then

\[d_{\underline{a},\underline{b}}=(\Delta_k^{a_k},\Delta_{k-1}^{a_{k-1}},\dots ,\Delta_1^{a_1}),\]
where $\Delta_j^{a_j}$ means that $\Delta_j$ appears $a_j$ times consecutively in the sequence, and 
$\Delta_j=2(a_1+a_2+\dots +a_{j-1})+a_j-n$ for $j=1,2,\dots ,n$. Moreover, $\Delta_i\neq \Delta_j$ for $i\neq j$ and $k \neq 2$.

\end{thm}

\begin{proof}

Since $\underline{b}=(n)$, clearly $b(i)=n+1-i$ for $i=1,2,\dots ,n$.

Let $j$ be any integer such that $1\leq j \leq k$. Let $i$ be any integer such that

\[1+a_{j+1}+a_{j+2}+\dots +a_k\leq i \leq a_j+a_{j+1}+\dots +a_k.\]
(Note that as
$j$ ranges from 1 to $k$, we are in fact considering all integers $i$ from 1 to $n$.)
We claim that $\sigma(i)=2(a_1+a_2+\dots +a_{j-1})+a_j-n+i$.

First we note that

\[b(i)=n+1-i\leq n+1-(1+a_{j+1}+a_{j+2}+\dots +a_k)=a_1+a_2+\dots +a_j,\]
and

\[b(i)=n+1-i\geq n+1-(a_j+a_{j+1}+\dots +a_k)=1+a_1+a_2+\dots a_{j-1}.\]
We therefore have

\[1+a_1+a_2+\dots a_{j-1}\leq b(i)\leq a_1+a_2+\dots +a_j.\]

\noindent
In other words, $b(i)$ is in the $j^{\text{th}}$ part of the top composition. It follows that $b(i)+t(b(i))$ is a constant equal to the first index plus the last index in the $j^{\text{th}}$ part of the top composition. Thus

\[b(i)+t(b(i))=(1+a_1+a_2+\dots a_{j-1})+(a_1+a_2+\dots +a_j).\]
Solving for $t(b(i))$, we have
%\[t(b(i))=-n-1+i+1+2(a_1+a_2+\dots a_{j-1})+a_j\]

\[t(b(i))=2(a_1+a_2+\dots a_{j-1})+a_j-n+i,\]
as claimed. The theorem follows directly from the claim.
\end{proof}

\begin{example}
The Frobenius seaweed $\mf{p}_9^\A \frac{2 \dd 3 \dd 4}{9}$ has the following meander (see left-hand side of Figure \ref{ThreeBlockJump}) with $\Delta_1 = 2$, which appears two times, $\Delta_2 = 7$, which appears three times, and $\Delta_3 = 5$, which appears four times. The top-bottom map defines, in the obvious way, a permutation on the set  $S=\{1,\dots,9   \}$ to yield the permutation cycle $\sigma_S=(4~9~2~7~5~3~8~1~6)$.
Now, define the mapping d, which gives the difference (mod 9) between consecutive elements of $\sigma_S$.  
We include a loop on a vertex when the top-map or the bottom-map is the identity on that vertex.  
\begin{figure}[H]
\begin{center}
    \begin{tabular}{c c}
$\begin{tikzpicture}[scale=.79, baseline=(current bounding box.center)]

\vertex (1) at (1,-5) {1};
\vertex (2) at (2,-5) {2};
\vertex (3) at (3,-5) {3};
\vertex (4) at (4,-5) {4};
\Loop[dist=1.7cm,dir=NO,style={in=20,out=160}](4)
\vertex (5) at (5,-5) {5};
\Loop[dist=1.7cm,dir=SO,style={in=-20,out=-160}](5)
\vertex (6) at (6,-5) {6};
\vertex (7) at (7,-5) {7};
\vertex (8) at (8,-5) {8};
\vertex (9) at (9,-5) {9};

\draw (1) to [bend left=50] (2);
\draw (3) to [bend left=100] (5);
\draw (6) to [bend left=50] (9);
\draw (7) to [bend left=50] (8);
\draw (1) to [bend right=100] (9);
\draw (2) to [bend right=100] (8);
\draw (3) to [bend right=100] (7);
\draw (4) to [bend right=100] (6);

\end{tikzpicture}$
&
\hspace{1.5cm}
\small{\begin{tabular}{c||c}
differences & $\Delta$'s \\
\hline
d$(4,9)$ & 5  \\
d$(9,2)$ & 2 \\
d$(2,7)$ & 5 \\
d$(7,5)$ & 7 \\
d$(5,3)$ & 7 \\
d$(3,8)$ & 5 \\
d$(8,1)$ & 2 \\
d$(1,6)$ & 5 \\
d$(6,4)$ & 7 \\
\end{tabular}}
\end{tabular}
\end{center}
\caption{$M_9^\A\frac{2 \dd 3 \dd 4}{9}$ and its chart of $\Delta$'s}
\label{ThreeBlockJump}
\end{figure}
\end{example}

When there are only two top parts, we have the following easy corollary.

\begin{cor} In Theorem \ref{meander deltas}, if $k=2$, then $\Delta_1 = \Delta_2 \equiv a_1 - n ~(\bmod ~n)$.  Here, we let $\Delta \equiv a_1 - n ~(\bmod ~n)$.  
\end{cor}
%\proof If $k=2$ in Theorem \ref{meander deltas}, then $c_1 = a_1 - n$ and $c_2 = 2a_1 + a_2 - n = a_1 - n$. \qed

\begin{example}
The Frobenius seaweed $\mf{p}_8^\A \frac{3 \dd 5}{8}$ has the following meander with $\Delta_1 = 3$ and $\Delta_2 = 3$ and permutation cycle $\sigma_S = (2~ 5~ 8~ 3~ 6~ 1~ 4~ 7)$.  Note that $\Delta=3$ ``generates" $\sigma_S$.  In this case, the mapping d gives the difference (mod 8) between consecutive elements of $\sigma_S$.

\begin{figure}[H]
\begin{center}
    \begin{tabular}{c c}
$\begin{tikzpicture}[scale=.79, baseline=(current bounding box.center)]

\vertex (1) at (1,0) {1};
\vertex (2) at (2,0) {2};
\vertex (3) at (3,0) {3};
\vertex (4) at (4,0) {4};
\vertex (5) at (5,0) {5};
\vertex (6) at (6,0) {6};
\vertex (7) at (7,0) {7};
\vertex (8) at (8,0) {8};

\draw (1) to [bend left=100] (3);
\draw (4) to [bend left=100] (8);
\draw (5) to [bend left=100] (7);
\Loop[dist=1.7cm,dir=NO,style={in=20,out=160}](2)
\draw (1) to [bend right=50] (8);
\draw (2) to [bend right=50] (7);
\Loop[dist=1.7cm,dir=NO,style={in=20,out=160}](6)
\draw (3) to [bend right=50] (6);
\draw (4) to [bend right=50] (5);

\end{tikzpicture}$
&
\hspace{1.5cm}
\small{\begin{tabular}{c||c}
differences & $\Delta$'s \\
\hline
d$(2,5)$ & 3  \\
d$(5,8)$ & 3 \\
d$(8,3)$ & 3 \\
d$(3,6)$ & 3 \\
d$(6,1)$ & 3 \\
d$(1,4)$ & 3 \\
d$(4,7)$ & 3 \\
d$(7,2)$ & 3 \\
\end{tabular}}
\end{tabular}
\end{center}
\caption{
$M_8^\A\frac{3 \dd 5}{8}$ and its chart of $\Delta$'s}
\label{TwoBlockJump}
\end{figure}

\end{example}

\begin{remark}
For Frobenius seaweeds $\mf{p}_n^\D\left(\dfrac{a \dd b}{c}, \III\right)$ with $a+b = n$ and $c<n$, we have 

\[\Delta_1 = \Delta_2 \equiv a-c ~(\bmod ~n).\] 
For the remainder of this paper, we will use $\Delta \equiv a-c ~(\bmod ~n)$ for such seaweeds.  
\end{remark}

We are now in a position to distinguish between the Frobenius and non-Frobenius cases described in, for example, Figures \ref{DFrobeniusH1} and \ref{NotFrobeniusH1}.  We find that coupling the necessary greatest common divisor condition with a congrunece relation will classify certain families of Frobenius seaweeds.

\begin{thm}\label{HomotopyTypeH1}
If $\mf{p}_n^\D\left(\dfrac{a \dd b}{c}, \III\right)$ has type-A homotopy type $H(1)$, then $\mf{p}_n^\D\left(\dfrac{a \dd b}{c}, \III\right)$ is Frobenius precisely when the following two conditions are met:   
\begin{enumerate}[\textup(i\textup)]
\item $\gcd(a+b,b+c) = 1$, and
\item  $0 < \displaystyle\frac{\Delta^{\varphi (n) - 1}}{n} - \lf \frac{\Delta^{\varphi (n) - 1}}{n} \rf$ $< 0.5$.  Here, $\varphi (n)$ is the Euler $\varphi$ function.  
\end{enumerate}
\end{thm}

\proof 
The first condition must be satisfied by previous observations.  

Let $\sigma$ be the permutation cycle for $\mf{p}_n^\A\dfrac{a \dd b}{c \dd 3}$.  Since $\Delta$ generates $\sigma$, there are distinct $k_1, k_2 \in (0,n)$ with 
\begin{eqnarray*}\label{tailofsize2proof}
n-1 +  k_1 \Delta &\equiv & n-2 ~(\bmod ~n), \\
n-1 + k_2 \Delta  &\equiv & 0 ~(\bmod ~n).
\end{eqnarray*}
If $k_1 > k_2$, then the path in the meander for $\mf{p}_n^\D\left(\dfrac{a \dd b}{c}, \III\right)$ containing $v_{n-1}$ also contains $v_n$.  In particular, if $k_1 > k_2$, then $v_{n-2}$ and $v_{n-1}$ are on different components in $\mf{p}_n^\D\left(\dfrac{a \dd b}{c}, \III\right)$, which, as noted earlier, is the second condition necessary for $\mf{p}_n^\D\left(\dfrac{a \dd b}{c}, \III\right)$ to be Frobenius.  These equations simplify to 
\begin{eqnarray}
k_1 \Delta  &\equiv & -1 ~(\bmod ~n), \\
k_2 \Delta  &\equiv & 1 ~(\bmod ~n). 
\end{eqnarray} 
\noindent
Multiplying equations (3) and (4) by $\Delta^{\varphi (n) - 1}$, and applying Euler's Totient theorem, yields the following system: 
\begin{eqnarray*}
k_1 &\equiv & -\Delta^{\varphi (n) - 1} ~(\bmod ~n), \\
k_2 &\equiv &~~ \Delta^{\varphi (n) - 1} ~(\bmod ~n). 
\end{eqnarray*} 
\noindent
If $k_1>k_2$, then $-\Delta^{\varphi (n) - 1} ~(\bmod ~n) > \Delta^{\varphi (n) - 1} ~(\bmod ~n)$, which is true when $0 \leq \Delta^{\varphi (n) - 1} ~(\bmod ~n) \leq \frac{n}{2}$.  But $k_i \neq 0$ for $i = 1,2$, so $\Delta^{\varphi (n) - 1} ~(\bmod ~n) \neq 0$.  Moreover, $k_1 \neq k_2$, so $\Delta^{\varphi (n) - 1} ~(\bmod ~n) \neq  \frac{n}{2}$.  The result follows.  \qed

\begin{example}
The Frobenius seaweed $\mf{p}_{10}^\D \left( \dfrac{4 \dd 6}{7}, \textrm{III}\right)$ has $\Delta = 7$.  With $\varphi (10) = 4$, we compute $\displaystyle\frac{\Delta ^{\varphi(n)-1}}{n} - \lf \frac{\Delta ^{\varphi(n)-1}}{n} \rf  = \frac{7^3}{10} - \lf \frac{7^3}{10} \rf = .3 <.5$.  

%permutation cycle $\sigma_S = (4~1~8~5~2~9~6~3~10~7)$.  

The seaweed $\mf{p}_{10}^\D \left( \dfrac{6 \dd 4}{7}, \textrm{III}\right)$ is not Frobenius.  For this seaweed, $\Delta = 9$.  With $\varphi (10) = 4$, we compute $\displaystyle\frac{\Delta ^{\varphi(n)-1}}{n} - \lf \frac{\Delta ^{\varphi(n)-1}}{n} \rf  = \frac{9^3}{10} - \lf \frac{9^3}{10} \rf = .9 >.5$

\end{example}

As a scholium of the proof of Theorem \ref{HomotopyTypeH1}, notice if a seaweed satisfies all hypotheses of Theorem \ref{HomotopyTypeH1}, but 
$.5 < \displaystyle\frac{\Delta^{\varphi (n) - 1}}{n} - \lf \frac{\Delta^{\varphi (n) - 1}}{n} \rf$ $< 1$, 
then it is not Frobenius.  

\noindent
\textbf{Case 3.2:  $|T|=4$}

When $|T|=4$ and $\mf{p}_n^\D\left(\dfrac{a \dd b}{c}, \III\right)$ is Frobenius, there are restrictions on the parts $a, b$, and $c$. 

\begin{lemma}
If $\mf{p}_n^\D\left(\dfrac{a \dd b}{c}, \III\right)$ is Frobenius, then $a, b,$ and $c$ are odd.  In particular, $n$ is even.  
\end{lemma}

\proof 
Such a meander contains four paths, and hence eight ends of paths.  Five of these are provided by the tail, so the other three must come from the parts.  
\qed 

\begin{thm}\label{Tail of size Four Lemma}
If $\mf{p}_n^\D\left(\dfrac{a \dd b}{c}, \III\right)$ is Frobenius, then $v_{n-2}$ and $v_n$ are on the same component.  
\end{thm}

\proof
If $v_n$ and $v_{n-1}$ are on the same component, then there are adjacent vertices connected by a single path, which contradicts that the blocks are of odd size.  

If $v_n$ and $v_{n-3}$ are on the same component, then there are four consecutive vertices among $v_1, ..., v_{n-6}$ with an edge joining the first and the fourth while the second and third vertices are not connected.  This cannot happen.  

If $v_n$ and $v_{n-4}$ are on the same component, then applying moves from Lemma \ref{WindingDown} to the meander $M_n^\A\dfrac{a \dd b}{c \dd 5}$ leaves the ``meander", which is not a valid homotopy type.  See Figure \ref{smile}.  

\begin{figure}[H]
\[\begin{tikzpicture}
\def\Node{\node [circle, fill, inner sep=2pt]}
\draw (0,0) circle (.75cm);

\Node at (-.25,0){};
\Node at (.25,0){};
\end{tikzpicture}\]
\caption{The wound down meander $M_n^\A\dfrac{a \dd b}{c \dd 5}$}
\label{smile}
\end{figure}
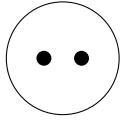
\qed

As a corollary of Theorem \ref{Tail of size Four Lemma}, if the seaweed $\mf{p}_n^\D\left(\dfrac{a \dd b}{c}, \III\right)$ is Frobenius, it must have type-A homotopy type $H(1,1)$.  
Such seaweeds necessarily have $\gcd(a+b,b+c)=2$; however, this does not provide a sufficient characterization as Figures \ref{HomotopyTypeH11Frob} and \ref{HomotopyTypeH11NotFrob} illustrate.  
%As before, the meander without the dotted lower arcs is a type-D meander, while the meander with the dotted lower arcs included is the type-A meander from which the type-A homotopy type is discerned.  

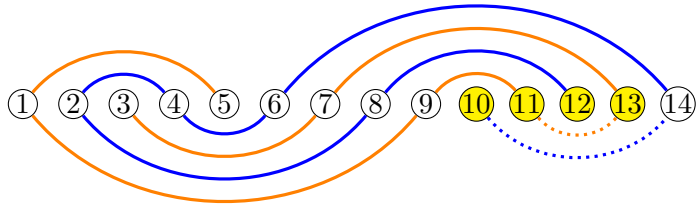
\begin{figure}[H]
\[\begin{tikzpicture}[scale=.67]

\vertex (1) at (1,0) {1};
\vertex (2) at (2,0) {2};
\vertex (3) at (3,0) {3};
\vertex (4) at (4,0) {4};
\vertex (5) at (5,0) {5};
\vertex(6) at (6,0) {6};
\vertex (7) at (7,0) {7};
\vertex (8) at (8,0) {8};
\vertex (9) at (9,0) {9};
\vertex[fill=yellow] (10) at (10,0) {10};
\vertex[fill=yellow] (11) at (11,0) {11};
\vertex[fill=yellow]  (12) at (12,0) {12};
\vertex[fill=yellow] (13) at (13,0) {13};
\vertex (14) at (14,0) {14};

\draw[color=orange, line width=1.2 pt] (1) to [bend left=50] (5);
\draw[color=blue, line width=1.2 pt] (2) to [bend left=50] (4);
\draw[color=blue, line width=1.2 pt] (6) to [bend left=50] (14);
\draw[color=orange, line width=1.2 pt] (7) to [bend left=50] (13);
\draw[color=blue, line width=1.2 pt] (8) to [bend left=50] (12);
\draw[color=orange, line width=1.2 pt] (9) to [bend left=50] (11);
\draw[color=orange, line width=1.2 pt] (1) to [bend right=50] (9);
\draw[color=blue, line width=1.2 pt] (2) to [bend right=50] (8);
\draw[color=orange, line width=1.2 pt] (3) to [bend right=50] (7);
\draw[color=blue, line width=1.2 pt] (4) to [bend right=50] (6);
\draw[dotted, color=blue, line width=1.2 pt] (10) to [bend right=50] (14);
\draw[dotted, color=orange, line width=1.2 pt] (11) to [bend right=50] (13);

;\end{tikzpicture}\]
\caption{The seaweed $\mf{p}_{14}^\D \left( \dfrac{5 \dd 9}{9}, \III\right)$ has type-A homotopy type $H(1,1)$ and index zero.}
\label{HomotopyTypeH11Frob}
\end{figure}

\begin{figure}[H]
\[\begin{tikzpicture}[scale=.59]

\vertex (1) at (1,0) {1};
\vertex (2) at (2,0) {2};
\vertex (3) at (3,0) {3};
\vertex (4) at (4,0) {4};
\vertex (5) at (5,0) {5};
\vertex (6) at (6,0) {6};
\vertex (7) at (7,0) {7};
\vertex (8) at (8,0) {8};
\vertex (9) at (9,0) {9};
\vertex (10) at (10,0) {10};
\vertex (11) at (11,0) {11};
\vertex (12) at (12,0) {12};
\vertex (13) at (13,0) {13};
\vertex (14) at (14,0) {14};
\vertex (15) at (15,0) {15};
\vertex (16) at (16,0) {16};
\vertex (17) at (17,0) {17};
\vertex[fill=yellow] (18) at (18,0) {18};
\vertex[fill=yellow] (19) at (19,0) {19};
\vertex[fill=yellow] (20) at (20,0) {20};
\vertex[fill=yellow] (21) at (21,0) {21};
\vertex (22) at (22,0) {22};

\draw[color=orange, line width=1.2 pt] (1) to [bend left=50] (9);
\draw[color=blue, line width=1.2 pt] (2) to [bend left=50] (8);
\draw[color=orange, line width=1.2 pt] (3) to [bend left=50] (7);
\draw[color=blue, line width=1.2 pt] (4) to [bend left=50] (6);
\draw[color=blue, line width=1.2 pt] (10) to [bend left=50] (22);
\draw[color=orange, line width=1.2 pt] (11) to [bend left=50] (21);
\draw[color=blue, line width=1.2 pt] (12) to [bend left=50] (20);
\draw[color=orange, line width=1.2 pt] (13) to [bend left=50] (19);
\draw[color=blue, line width=1.2 pt] (14) to [bend left=50] (18);
\draw[color=orange, line width=1.2 pt] (15) to [bend left=50] (17);
\draw[color=orange, line width=1.2 pt] (1) to [bend right=50] (17);
\draw[color=blue, line width=1.2 pt] (2) to [bend right=50] (16);
\draw[color=orange, line width=1.2 pt] (3) to [bend right=50] (15);
\draw[color=blue, line width=1.2 pt] (4) to [bend right=50] (14);
\draw[color=orange, line width=1.2 pt] (5) to [bend right=50] (13);
\draw[color=blue, line width=1.2 pt] (6) to [bend right=50] (12);
\draw[color=orange, line width=1.2 pt] (7) to [bend right=50] (11);
\draw[color=blue, line width=1.2 pt] (8) to [bend right=50] (10);
\draw[dotted, color=blue, line width=1.2 pt] (18) to [bend right=50] (22);
\draw[dotted, color=orange, line width=1.2 pt] (19) to [bend right=50] (21);

;\end{tikzpicture}\]
\caption{The seaweed $\mf{p}_{22}^\D \left( \dfrac{9 \dd 13}{17}, \III\right)$ has type-A homotopy type $H(1,1)$ and index two.}
\label{HomotopyTypeH11NotFrob}
\end{figure}
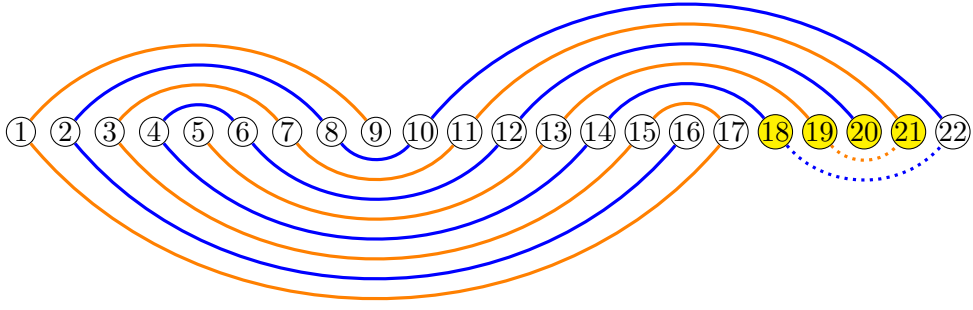

To differentiate between these, we make the following observations about the type-A meanders from the previous examples:

\begin{enumerate}

\item Each meander consists of two paths, a blue path and an orange path; the blue path spans the even vertices, and the orange path spans the odd vertices. 

\item In each meander, the orange path contains two components of the meander for $\mf{p}_{n}^\D \left( \dfrac{a \dd b}{c}, \III\right)$, each of which is a path with one end in the tail and contributes $0$ to the index of $\mf{p}_{n}^\D \left( \dfrac{a \dd b}{c}, \III\right)$. 

\item In the meander in Figure \ref{HomotopyTypeH11Frob}, the blue path contains two components of the meander for the seaweed $\mf{p}_{n}^\D \left( \dfrac{a \dd b}{c}, \III\right)$, each of which is a path with one end in the tail and contributes zero to the index of $\mf{p}_{n}^\D \left( \dfrac{a \dd b}{c}, \III\right)$.  However, the same does not hold for the meander in Figure \ref{HomotopyTypeH11NotFrob}.  Here, while the blue path does contain two components of the meander for $\mf{p}_{n}^\D \left( \dfrac{a \dd b}{c}, \III\right)$, one is a path with zero ends in the tail, and the other is a path with two ends in the tail; each of these components contributes one to the index of $\mf{p}_{n}^\D \left( \dfrac{a \dd b}{c}, \III\right)$. 

\item The top-bottom map gives two permutations associated to each meander: the first, $\sigma_1$ (the blue path), spanning the even vertices and a second, $\sigma_2$ (the orange path), spanning the odd vertices.  Each permutation is generated by $\Delta \equiv a-c ~(\bmod ~n)$ by Theorem \ref{meander deltas}.  
\end{enumerate}

\noindent
Combining observations $3$ and $4$ above, what will differentiate the Frobenius case from the non-Frobenius case must be captured by the blue path.  We will add to the previously-mentioned greatest common divisor condition an argument similar to the proof of Theorem \ref{HomotopyTypeH1} to obtain a sufficient condition.

\begin{thm}\label{HomotopyTypeH11}
The seaweed $\mf{p}_n^\D\left(\dfrac{a \dd b}{c}, \III\right)$
is Frobenius precisely when the following two conditions are met:

\begin{enumerate}[\textup(i\textup)]
\item $\gcd(a+b,b+c) = 2$, and
\item $0 < \displaystyle\frac{(\frac{\Delta}{2})^{\varphi (\frac{n}{2}) - 1}}{n} - \lf \frac{(\frac{\Delta}{2})^{\varphi (\frac{n}{2}) - 1}}{n} \rf$ $< .5$.  Here, $\varphi (n)$ is the Euler $\varphi$ function. 
\end{enumerate}
\end{thm} 

\proof
The first condition must be satisfied by previous observations.  

Let $\sigma_1$ be the permutation cycle spanning the even vertices in $\mf{p}_n^\A\dfrac{a \dd b}{c \dd 5}$.  Let $\sigma$ be the permutation of $\{1, ...,\frac{n}{2}\}$ obtained by dividing each entry of $\sigma_1$ by $2$.  Since $\Delta$ generates $\sigma_1$, $\frac{\Delta}{2}$ generates $\sigma$.  Thus there are distinct $k_1, k_2 \in (0,\frac{n}{2})$ with 
\begin{eqnarray*}\label{tailofsize2proof}
\frac{n}{2}-1 + k_1 \frac{\Delta}{2} &\equiv & \frac{n}{2}-2 ~\left(\bmod ~\frac{n}{2}\right), \\
\frac{n}{2}-1 + k_2 \frac{\Delta}{2} &\equiv & 0 ~\left(\bmod ~\frac{n}{2}\right). 
\end{eqnarray*}
\noindent
If $k_1 > k_2$, then the path in the meander for $\mf{p}_n^\D\left(\dfrac{a \dd b}{c}, \III\right)$ containing $v_{n-2}$ also contains $v_n$.  In particular, if $k_1 > k_2$, then  $\mf{p}_n^\D\left(\dfrac{a \dd b}{c}, \III\right)$, is Frobenius.  These equations simplify to 

%\begin{eqnarray}
%\Delta k_1 &\equiv & -2 ~(\bmod ~\displaystyle\frac{n}{2}) \\
%\Delta k_2 &\equiv & 2 ~(\bmod ~\displaystyle\frac{n}{2}). 
%\end{eqnarray} 

%\noindent
%Multiplying equations (3) and (4) by $\Delta^{\varphi (n) - 1}$ yields the following system.  Recall $k_i \in (0,\displaystyle\frac{n}{2})$.  
\begin{eqnarray*}
k_1 &\equiv & -\frac{\Delta}{2}^{\varphi (\frac{n}{2}) - 1} ~\left(\bmod ~\frac{n}{2}\right), \\
k_2 &\equiv & \frac{\Delta}{2}^{\varphi (\frac{n}{2}) - 1} ~\left(\bmod ~\frac{n}{2}\right). 
\end{eqnarray*} 
If $k_1>k_2$ and $0 \leq \frac{\Delta}{2}^{\varphi (\frac{n}{2}) - 1} ~(\bmod ~\frac{n}{2}) \leq \frac{n}{4}$, then $-\frac{\Delta}{2}^{\varphi (\frac{n}{2}) - 1} ~(\bmod ~\frac{n}{2}) > \frac{\Delta}{2}^{\varphi (\frac{n}{2}) - 1} ~(\bmod ~\frac{n}{2})$.  But $k_i \neq 0$ for $i = 1,2$, so $\Delta^{\varphi (\frac{n}{2}) - 1} ~(\bmod ~\frac{n}{2}) \neq 0$.  The result follows.  \qed

\begin{example}

The Frobenius seaweed $\mf{p}_{14}^\D \left( \dfrac{5 \dd 9}{9}, \textrm{III}\right)$ has $\frac{\Delta}{2} = 5$.  With $\varphi (7) = 6$, we compute $\displaystyle\frac{(\frac{\Delta}{2}) ^{\varphi(\frac{n}{2})-1}}{n} - \lf \frac{(\frac{\Delta}{2}) ^{\varphi(\frac{n}{2})-1}}{n} \rf  = \frac{5^5}{14} - \lf \frac{5^5}{14} \rf \approx .21 <.5$.  

%permutation cycle $\sigma_S = (4~1~8~5~2~9~6~3~10~7)$.  

The seaweed $\mf{p}_{22}^\D \left( \dfrac{9 \dd 13}{17}, \textrm{III}\right)$ is not Frobenius.  For this seaweed, $\frac{\Delta}{2} = 7$.  With $\varphi (11) = 10$, we compute $\displaystyle\frac{(\frac{\Delta}{2}) ^{\varphi(\frac{n}{2})-1}}{n} - \lf \frac{(\frac{\Delta}{2}) ^{\varphi(\frac{n}{2})-1}}{n} \rf  = \frac{7^{9}}{22} - \lf \frac{7^{9}}{22} \rf \approx .86 >.5$
\end{example}

As a scholium of the proof of Theorem \ref{HomotopyTypeH11}, notice that if a seaweed satisfies every condition in Theorem \ref{HomotopyTypeH11}, but 
$.5 < \displaystyle\frac{(\frac{\Delta}{2})^{\varphi (n) - 1}}{n} - \lf \frac{(\frac{\Delta}{2})^{\varphi (n) - 1}}{n} \rf$ $< 1$, 
then it is not Frobenius.

We now consider seaweeds of the form $\mf{p}_n^\D\frac{a \dd b \dd c}{d}$ with $a+b+c=n$.  If $d=n$, we obtain an index formula as a corollary to type-A results.  But for $d<n$, we leverage Theorem \ref{5 parts} by Karnauhova and Liebscher to show that such seaweeds not only have no linear gcd formula for their index, but also no polynomial gcd formula for their index, regardless of tail configuration.  

\begin{thm}\label{KarLiebTypeD}
Consider the seaweed $\mf{p}=\mf{p}_n^\D\frac{a \dd b \dd c}{d}$ with $a+b+c=n$.  

\begin{enumerate}[\textup(i\textup)]
\item If $d=n$, then $\ind \mf{p} = \gcd(a+b,b+c)$.  
\item If $d<n$, then there do not exist homogeneous polynomials $f_1,f_2\in \Z[x_1,x_2,x_3]$ of arbitrary degree such that $\ind \mf{p}$ is given by 
$\gcd(f_1(a,b,c),f_2(a,b,c))$.
\end{enumerate}
\end{thm}

\proof
Let $M$ be the meander for $\mf{p}$.  If $d=n$, then $M$ contains zero tail vertices, so by Theorem \ref{typeD}, $\ind \mf{p} = 2C + P$, where $P$ is the number of paths in $M$.  Consider the seaweed $\mf{p}_1=\mf{p}_n^\A\frac{a \dd b \dd c}{d}$.  By Theorem \ref{A 4 parts}, $\ind \mf{p}_1 = \gcd(a+b,b+c)-1$.  By Theorem \ref{DKformula}, $\ind \mf{p}_1 = 2C+P-1$.  Hence $\gcd(a+b,b+c) = 2C+P$.  But the meander for $\mf{p}_1$ is isomorphic to $M$, so $\ind \mf{p} = \gcd (a+b,b+c)$.

If $d<n$, then the index computations for $M$ have the same complexity as the index computations for $\mf{p}^A_{2n-d} \frac{n-d \dd a \dd b \dd c}{2n-d}$, which by Theorem \ref{5 parts} has no polynomial gcd formula for its index.  
\qed 

\begin{example}
The black seaweed in Figure \ref{DFourBlocks} is a four-part type-D seaweed which can be extended to an arbitrary five-part type-A seaweed whose connected components cannot be counted by a polynomial gcd formula.
Intuitively, a type-D seaweed with four total parts corresponds to a five-part type-A seaweed using a construction similar to that in Figure \ref{DFourBlocks}.
%, as in Figure \ref{DFourBlocks}

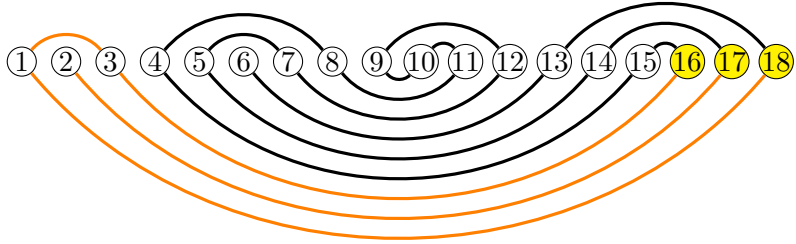
\begin{figure}[H]
\[\begin{tikzpicture}[scale=.59]

\vertex (1) at (1,0) {1};
\vertex (2) at (2,0) {2};
\vertex (3) at (3,0) {3};
\vertex (4) at (4,0) {4};
\vertex (5) at (5,0) {5};
\vertex (6) at (6,0) {6};
\vertex (7) at (7,0) {7};
\vertex (8) at (8,0) {8};
\vertex (9) at (9,0) {9};
\vertex (10) at (10,0) {10};
\vertex (11) at (11,0) {11};
\vertex (12) at (12,0) {12};
\vertex (13) at (13,0) {13};
\vertex (14) at (14,0) {14};
\vertex (15) at (15,0) {15};
\vertex[fill=yellow] (16) at (16,0) {16};
\vertex[fill=yellow] (17) at (17,0) {17};
\vertex[fill=yellow] (18) at (18,0) {18};

\draw[color=orange, line width=1.2 pt] (1) to [bend left=50] (3);
\draw[color=black, line width=1.2 pt] (4) to [bend left=50] (8);
\draw[color=black, line width=1.2 pt] (5) to [bend left=50] (7);
\draw[color=black, line width=1.2 pt] (9) to [bend left=50] (12);
\draw[color=black, line width=1.2 pt] (10) to [bend left=50] (11);
\draw[color=black, line width=1.2 pt] (13) to [bend left=50] (18);
\draw[color=black, line width=1.2 pt] (14) to [bend left=50] (17);
\draw[color=black, line width=1.2 pt] (15) to [bend left=50] (16);
\draw[color=orange, line width=1.2 pt] (1) to [bend right=50] (18);
\draw[color=orange, line width=1.2 pt] (2) to [bend right=50] (17);
\draw[color=orange, line width=1.2 pt] (3) to [bend right=50] (16);
\draw[color=black, line width=1.2 pt] (4) to [bend right=50] (15);
\draw[color=black, line width=1.2 pt] (5) to [bend right=50] (14);
\draw[color=black, line width=1.2 pt] (6) to [bend right=50] (13);
\draw[color=black, line width=1.2 pt] (7) to [bend right=50] (12);
\draw[color=black, line width=1.2 pt] (8) to [bend right=50] (11);
\draw[color=black, line width=1.2 pt] (9) to [bend right=50] (10);

;\end{tikzpicture}\]
\caption{The seaweeds $\mf{p}_{15}^\D \dfrac{5 \dd 4 \dd 6}{12}$ and $\mf{p}_{18}^\A \dfrac{3 \dd 5 \dd 4 \dd 6}{18}$} 
\label{DFourBlocks}
\end{figure}
\end{example}

\subsection{Type-D seaweeds without seaweed shape}

Finally, we analyze type-D seaweeds without seaweed shape.  Recall that the classification from Theorem \ref{seaweed shape} and Figure \ref{DSeaweedWithoutSeaweedShpe} provide a useful visual for such seaweeds.  
To compute the index of seaweeds without seaweed shape, we will make a specific switch in which simple roots define the seaweed, yielding a new seaweed which does have seaweed shape.  The index of the original seaweed is either the same as the index of the new seaweed or the index of the new seaweed minus two. As a corollary of Theorem 4.1 in \textbf{\cite{Panyushev3}}, we have the following theorem.  

\begin{thm}\label{morphism}
Let $\mf{p}_1 = \mf{p}_n^\D(\Psi \dd \Psi')$ be a seaweed without seaweed shape with $\alpha_{n-1} \in \Psi \setminus \Psi '$.  Let $\Psi '' = \Psi ' - \alpha_n + \alpha_{n-1}$.  Then $\mf{p}_2 = \mf{p}_n^\D(\Psi \dd \Psi '') = \mf{p}_n^\D ((a_1, ..., a_m) \dd (b_1, ..., b_l))$ has seaweed shape.  Let $M$ be the meander associated to $\mf{p}_2$, and consider vertices $v_{n-a_m+1}$ and $v_n$ in $M$.  Then
\begin{enumerate}[\textup(i\textup)]
    \item $\ind \mf{p}_1 = \ind \mf{p}_2$ if $v_{n-a_m+1}$ and $v_n$ are on a path, and
    \item $\ind \mf{p}_1 = \ind \mf{p}_2 -2$ if $v_{n-a_m+1}$ and $v_n$ are on a cycle.
\end{enumerate}
\end{thm}

\proof 
This follows since $M$ has no tail vertices.  
\qed

%\proof
%That $\mf {p}_2$ has seaweed shape is a corollary of Theorem \ref{seaweed shape}. 
%It follows that the standard form for 
%$\mf{p}_2$ may be described by two full compositions of $n$; name them $(a_1, ..., a_k)$ and $(b_1, ..., b_l)$.  Assume, without loss of generality, that $a_k \leq b_l$.  Let $M$ be the meander associated with $\mf{p}_2$.  The meander $M$ contains ``distinguished" vertices $v_{n-a_k+1}$ and $v_n$.  If these vertices are on a path in $M$, then $\ind\mf{p}_1 = \ind\mf{p}_2 + 1$.  If they are on a cycle, then $\ind\mf{p}_1 = \ind\mf{p}_2 - 1$. 
%\qed 

\begin{example}
We illustrate the two cases of Theorem \ref{morphism} in Figures \ref{Dswitch} and \ref{Dswitch2}, respectively.

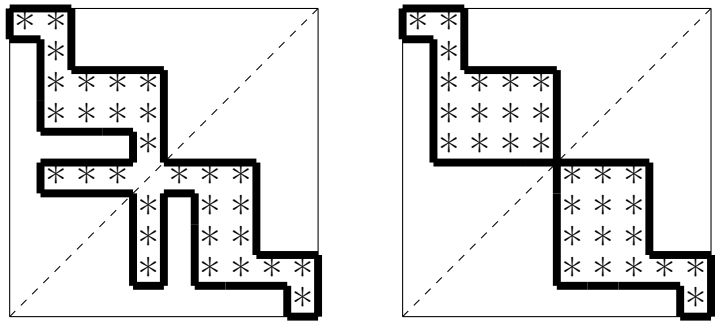
\begin{figure}[H]
\[\begin{tikzpicture}[scale=.41]
\draw (0,0) -- (0,10);
\draw (0,10) -- (10,10);
\draw (10,10) -- (10,0);
\draw (10,0) -- (0,0);

\draw [line width=3](0,10) -- (2,10);
\draw [line width=3](2,10) -- (2,8);
\draw [line width=3](2,8) -- (5,8);
\draw [line width=3](5,8) -- (5,5);
\draw [line width=3](5,5) -- (8,5);
\draw [line width=3](8,5) -- (8,2);
\draw [line width=3](8,2) -- (10,2);
\draw [line width=3](10,2) -- (10,0);

\draw [line width=3](0,10) -- (0,9);
\draw [line width=3](0,9) -- (1,9);
\draw [line width=3](1,9) -- (1,7);
\draw [line width=3](1,6) -- (3,6);
\draw [line width=3](1,7) -- (1,6);
\draw [line width=3](3,6) -- (4,6);
\draw [line width=3](4,6) -- (4,5);
%\draw [line width=3](4,5) -- (5,5);
%\draw [line width=3](5,5) -- (5,4);
\draw [line width=3](5,4) -- (6,4);
\draw [line width=3](6,4) -- (6,3);
\draw [line width=3](6,3) -- (6,1);
\draw [line width=3](6,1) -- (7,1);
\draw [line width=3](7,1) -- (9,1);
\draw [line width=3](9,1) -- (9,0);
\draw [line width=3](9,0) -- (10,0);
\draw [line width=3](1,4) -- (1,5);
\draw [line width=3](1,5) -- (4,5);
%\draw [line width=3](4,5) -- (4,4);
\draw [line width=3](4,4) -- (1,4);
\draw [line width=3](4,1) -- (5,1);
\draw [line width=3](5,1) -- (5,4);
%\draw [line width=3](5,4) -- (4,4);
\draw [line width=3](4,4) -- (4,1);

\draw [dashed] (0,0) -- (10,10);

\node at (0.5,9.4) {{\LARGE *}};
\node at (1.5,9.4) {{\LARGE *}};
\node at (1.5,8.4) {{\LARGE *}};
\node at (1.5,7.4) {{\LARGE *}};
\node at (2.5,7.4) {{\LARGE *}};
\node at (3.5,7.4) {{\LARGE *}};
\node at (4.5,7.4) {{\LARGE *}};
\node at (1.5,6.4) {{\LARGE *}};
\node at (2.5,6.4) {{\LARGE *}};
\node at (3.5,6.4) {{\LARGE *}};
\node at (4.5,6.4) {{\LARGE *}};
\node at (4.5,5.4) {{\LARGE *}};
\node at (1.5,4.4) {{\LARGE *}};
\node at (2.5,4.4) {{\LARGE *}};
\node at (3.5,4.4) {{\LARGE *}};
\node at (5.5,4.4) {{\LARGE *}};
\node at (6.5,4.4) {{\LARGE *}};
\node at (7.5,4.4) {{\LARGE *}};
\node at (4.5,3.4) {{\LARGE *}};
\node at (6.5,3.4) {{\LARGE *}};
\node at (7.5,3.4) {{\LARGE *}};
\node at (4.5,2.4) {{\LARGE *}};
\node at (6.5,2.4) {{\LARGE *}};
\node at (7.5,2.4) {{\LARGE *}};
\node at (4.5,1.4) {{\LARGE *}};
\node at (6.5,1.4) {{\LARGE *}};
\node at (7.5,1.4) {{\LARGE *}};
\node at (8.5,1.4) {{\LARGE *}};
\node at (9.5,1.4) {{\LARGE *}};
\node at (9.5,0.4) {{\LARGE *}};

\end{tikzpicture}
\hspace{1cm}
\begin{tikzpicture}[scale=.41]
\draw (0,0) -- (0,10);
\draw (0,10) -- (10,10);
\draw (10,10) -- (10,0);
\draw (10,0) -- (0,0);

\draw [line width=3](0,10) -- (2,10);
\draw [line width=3](2,10) -- (2,8);
\draw [line width=3](2,8) -- (5,8);
\draw [line width=3](5,8) -- (5,5);
\draw [line width=3](5,5) -- (8,5);
\draw [line width=3](8,5) -- (8,2);
\draw [line width=3](8,2) -- (10,2);
\draw [line width=3](10,2) -- (10,0);

\draw [line width=3](0,10) -- (0,9);
\draw [line width=3](0,9) -- (1,9);
\draw [line width=3](1,9) -- (1,5);
\draw [line width=3](1,5) -- (4,5);
\draw [line width=3](4,5) -- (5,5);
\draw [line width=3](5,5) -- (5,4);
\draw [line width=3](6,1) -- (7,1);
\draw [line width=3](7,1) -- (9,1);
\draw [line width=3](9,1) -- (9,0);
\draw [line width=3](9,0) -- (10,0);
\draw [line width=3](5,1) -- (5,4);
\draw [line width=3](5,1) -- (6,1);

\draw [dashed] (0,0) -- (10,10);

\node at (0.5,9.4) {{\LARGE *}};
\node at (1.5,9.4) {{\LARGE *}};
\node at (1.5,8.4) {{\LARGE *}};
\node at (1.5,7.4) {{\LARGE *}};
\node at (2.5,7.4) {{\LARGE *}};
\node at (3.5,7.4) {{\LARGE *}};
\node at (4.5,7.4) {{\LARGE *}};
\node at (1.5,6.4) {{\LARGE *}};
\node at (2.5,6.4) {{\LARGE *}};
\node at (3.5,6.4) {{\LARGE *}};
\node at (4.5,6.4) {{\LARGE *}};
\node at (4.5,5.4) {{\LARGE *}};
\node at (1.5,5.4) {{\LARGE *}};
\node at (2.5,5.4) {{\LARGE *}};
\node at (3.5,5.4) {{\LARGE *}};
\node at (5.5,4.4) {{\LARGE *}};
\node at (6.5,4.4) {{\LARGE *}};
\node at (7.5,4.4) {{\LARGE *}};
\node at (5.5,3.4) {{\LARGE *}};
\node at (6.5,3.4) {{\LARGE *}};
\node at (7.5,3.4) {{\LARGE *}};
\node at (5.5,2.4) {{\LARGE *}};
\node at (6.5,2.4) {{\LARGE *}};
\node at (7.5,2.4) {{\LARGE *}};
\node at (5.5,1.4) {{\LARGE *}};
\node at (6.5,1.4) {{\LARGE *}};
\node at (7.5,1.4) {{\LARGE *}};
\node at (8.5,1.4) {{\LARGE *}};
\node at (9.5,1.4) {{\LARGE *}};
\node at (9.5,0.4) {{\LARGE *}};

\end{tikzpicture}
%\hspace{1cm}
%\begin{tikzpicture}[scale=.67]
%[decoration={markings,mark=at position 0.6 with 
%{\arrow{angle 90}{>}}}]
%\vertex (1) at (1,0) {1};
%\vertex (2) at (2,0) {2};
%\vertex[fill=orange] (3) at (3,0) {3};
%\vertex (4) at (4,0) {4};
%\vertex[fill=orange] (5) at (5,0) {5};
%\draw (1) to [bend left=50] (2);
%\draw (3) to [bend left=50] (5);
%\draw (2) to [bend right=50] (5);
%\draw (3) to [bend right=50] (4);
%;\end{tikzpicture}
\]

\caption{The index of $\mf{p}_5^\D(
\{\alpha_2,\alpha_3,\alpha_5\} \dd \{\alpha_1,\alpha_3,\alpha_4\})$ (left) and 
$\mf{p}_5^\D(
\{\alpha_2,\alpha_3,\alpha_4\} \dd \{\alpha_1,\alpha_3,\alpha_4\})$ (right) are both one.}
\label{Dswitch}
\end{figure}

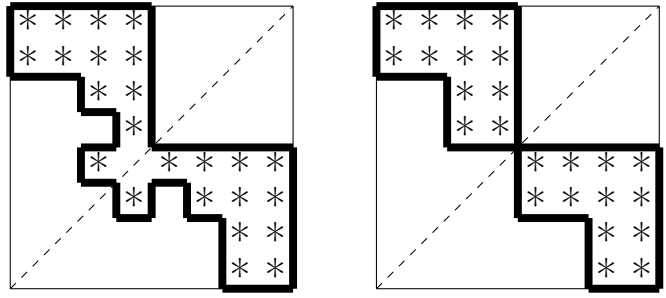
\begin{figure}[H]
\[\begin{tikzpicture}[scale=.47]
\draw (0,0) -- (0,8);
\draw (0,8) -- (8,8);
\draw (8,8) -- (8,0);
\draw (8,0) -- (0,0);

\draw [line width=3](0,8) -- (0,6);
\draw [line width=3](0,6) -- (2,6);
\draw [line width=3](2,6) -- (2,5);
\draw [line width=3](2,5) -- (3,5);
\draw [line width=3](3,5) -- (3,4);
\draw [line width=3](3,4) -- (2,4);
\draw [line width=3](2,4) -- (2,3);
\draw [line width=3](2,3) -- (3,3);
\draw [line width=3](3,3) -- (3,2);
\draw [line width=3](3,2) -- (4,2);
\draw [line width=3](4,2) -- (4,3);
\draw [line width=3](4,3) -- (5,3);
\draw [line width=3](5,3) -- (5,2);
\draw [line width=3](5,2) -- (6,2);
\draw [line width=3](6,2) -- (6,0);
\draw [line width=3](6,0) -- (8,0);

\draw [line width=3](0,8) -- (4,8);
\draw [line width=3](4,8) -- (4,4);
\draw [line width=3](4,4) -- (8,4);
\draw [line width=3](8,4) -- (8,0);

\draw [dashed] (0,0) -- (8,8);

\node at (.5,7.4) {{\LARGE *}};
\node at (1.5,7.4) {{\LARGE *}};
\node at (2.5,7.4) {{\LARGE *}};
\node at (3.5,7.4) {{\LARGE *}};
\node at (.5,6.4) {{\LARGE *}};
\node at (1.5,6.4) {{\LARGE *}};
\node at (2.5,6.4) {{\LARGE *}};
\node at (3.5,6.4) {{\LARGE *}};
\node at (2.5,5.4) {{\LARGE *}};
\node at (3.5,5.4) {{\LARGE *}};
\node at (3.5,4.4) {{\LARGE *}};
\node at (2.5,3.4) {{\LARGE *}};
\node at (4.5,3.4) {{\LARGE *}};
\node at (5.5,3.4) {{\LARGE *}};
\node at (6.5,3.4) {{\LARGE *}};
\node at (7.5,3.4) {{\LARGE *}};
\node at (3.5,2.4) {{\LARGE *}};
\node at (5.5,2.4) {{\LARGE *}};
\node at (6.5,2.4) {{\LARGE *}};
\node at (7.5,2.4) {{\LARGE *}};
\node at (6.5,1.4) {{\LARGE *}};
\node at (7.5,1.4) {{\LARGE *}};
\node at (6.5,.4) {{\LARGE *}};
\node at (7.5,.4) {{\LARGE *}};

\end{tikzpicture}
\hspace{1cm}
\begin{tikzpicture}[scale=.47]
\draw (0,0) -- (0,8);
\draw (0,8) -- (8,8);
\draw (8,8) -- (8,0);
\draw (8,0) -- (0,0);

\draw [line width=3](0,8) -- (0,6);
\draw [line width=3](0,6) -- (2,6);
\draw [line width=3](2,6) -- (2,4);
\draw [line width=3](2,4) -- (4,4);
\draw [line width=3](4,4) -- (4,2);
\draw [line width=3](4,2) -- (6,2);
\draw [line width=3](6,2) -- (6,0);
\draw [line width=3](6,0) -- (8,0);

\draw [line width=3](0,8) -- (4,8);
\draw [line width=3](4,8) -- (4,4);
\draw [line width=3](4,4) -- (8,4);
\draw [line width=3](8,4) -- (8,0);

\draw [dashed] (0,0) -- (8,8);

\node at (.5,7.4) {{\LARGE *}};
\node at (1.5,7.4) {{\LARGE *}};
\node at (2.5,7.4) {{\LARGE *}};
\node at (3.5,7.4) {{\LARGE *}};
\node at (.5,6.4) {{\LARGE *}};
\node at (1.5,6.4) {{\LARGE *}};
\node at (2.5,6.4) {{\LARGE *}};
\node at (3.5,6.4) {{\LARGE *}};
\node at (2.5,5.4) {{\LARGE *}};
\node at (3.5,5.4) {{\LARGE *}};
\node at (2.5,4.4) {{\LARGE *}};
\node at (3.5,4.4) {{\LARGE *}};
\node at (4.5,2.4) {{\LARGE *}};
\node at (5.5,3.4) {{\LARGE *}};
\node at (6.5,3.4) {{\LARGE *}};
\node at (7.5,3.4) {{\LARGE *}};
\node at (4.5,3.4) {{\LARGE *}};
\node at (5.5,2.4) {{\LARGE *}};
\node at (6.5,2.4) {{\LARGE *}};
\node at (7.5,2.4) {{\LARGE *}};
\node at (6.5,1.4) {{\LARGE *}};
\node at (7.5,1.4) {{\LARGE *}};
\node at (6.5,.4) {{\LARGE *}};
\node at (7.5,.4) {{\LARGE *}};

\end{tikzpicture}
%\hspace{1cm}
%\begin{tikzpicture}[scale=.67]
%[decoration={markings,mark=at position 0.6 with 
%{\arrow{angle 90}{>}}}]
%\vertex (1) at (1,0) {1};
%\vertex (2) at (2,0) {2};
%\vertex[fill=orange] (3) at (3,0) {3};
%\vertex[fill=orange] (4) at (4,0) {4};
%\draw (1) to [bend left=50] (4);
%\draw (2) to [bend left=50] (3);
%\draw (1) to [bend right=50] (2);
%\draw (3) to [bend right=50] (4);
%;\end{tikzpicture}
\]

\caption{The index of $\mf{p}_4^\D(
\{\alpha_1,\alpha_4\} \dd \{\alpha_1,\alpha_2,\alpha_3\})$ (left) is zero, and the index of
$\mf{p}_4^\D(
\{\alpha_1,\alpha_3\} \dd \{\alpha_1,\alpha_2,\alpha_3\})$ (right) is two.}
\label{Dswitch2}
\end{figure}
\end{example}

As a corollary of Theorem \ref{morphism}, we can classify Frobenius type-D seaweeds without seaweed shape.  

\begin{thm}\label{FrobeniusWithoutSeaweedShape}
There is a bijection between Frobenius type-D seaweeds without seaweed shape and type-A seaweeds with homotopy type $H(2)$.  
\end{thm}

\proof
Following the notation of Theorem \ref{morphism}, since $\ind \mf{p}_2 \geq 0$, the only way $\ind \mf{p}_1$ can equal zero is if distinguished vertices $v_{n-a_k+1}$ and $v_n$ are on a cycle and $\ind \mf{p}_2 = 2$.  The only such $\mf{p}_2$ have a meander graph consisting of exactly one cycle.  Moreover, the parts defining $\mf{p}_2$ also define a type-A seaweed, and consequently, a type-A meander identical to $M$.  Such a type-A seaweed has homotopy type $H(2)$.  
\qed

%The authors are grateful to Murray Gerstenhaber, Tony Giaquinto,  and Jim Stasheff for a number of helpful discussions.

%%%%%%%%%%%%%%%%%%%%%%%%%%%%%%%%%%%%%%%%%%%%%%%%%%%%%%%%%

\section{Appendix A - The signature and homotopy type of a meander}

The following lemma establishes that, using a deterministic sequence of graph-theoretic moves, each meander can be contracted or ``wound down" to the empty meander, a meander with no vertices.  
The sequence of moves applied to a meander is called the \textit{signature} of the meander, and the meander's \textit{homotopy type} can then be read off of the signature.  

\begin{lemma}[Coll et al. \textbf{\cite{Coll3}}, Lemma 4.1]\label{WindingDown} Given the meander $M=M_n^\A\frac{a_1|a_2|...|a_m}{b_1|b_2|...|b_t}$, a new meander $M'$ can be created by one of the following moves:
\begin{enumerate}[\textup(i\textup)]
\item {\bf Flip $(F)$:} If $a_1<b_1$, then 
$M'=M_n^\A\frac{b_1|b_2|...|b_t}{a_1|a_2|...|a_m}$,
\item {\bf Component Elimination $(C(c))$:} 
If $a_1=b_1=c$, then  $M'=M_n^\A\frac{a_2|a_3|...|a_m}{b_2|b_3|...|b_t},$
\item {\bf Rotation Contraction $(R)$:} If $b_1<a_1<2b_1$, then 
$M'=M_n^\A\frac{b_1|a_2|a_3|...|a_m}{(2b_1-a_1)|b_2|...|b_t}$,
\item {\bf Block Elimination $(B)$:} If $a_1=2b_1$, then 
$M'=M_n^\A\frac{b_1|a_2|..|a_m}{b_2|b_3|...|b_t}$,
\item {\bf Pure Contraction $(P)$:} If $a_1>2b_1$, then 
$M'=M_n^\A
\frac{(a_1-2b_1)|b_1|a_2|a_3|...|a_m}{b_2|b_3|...|b_t}$.
\end{enumerate}
\end{lemma}

This winding down process is illustrated by the following example.

\begin{example}
We find the signature of the meander $M_{10}^\A\frac{3 \dd 7}{2 \dd 5 \dd 3}$ in Figure \ref{WindingDownSignature}.

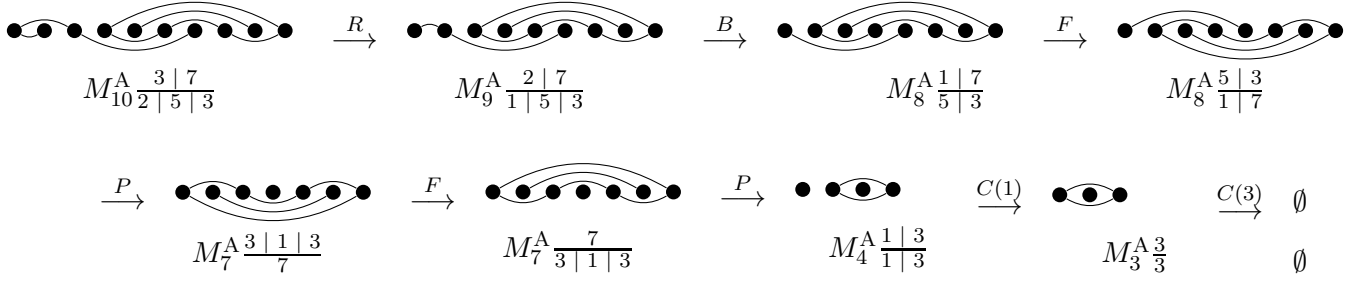
\begin{figure}[H]
$$\hspace*{-.75cm}\begin{tikzpicture}[scale=.4]
\def\Node{\node [circle, fill, inner sep=2pt]}
\node at (5.5,-2){$M_{10}^\A\frac{3 \dd 7}{2 \dd 5 \dd 3}$};
\Node (1) at (1,0){};
\Node (2) at (2,0){};
\Node (3) at (3,0){};
\Node (4) at (4,0){};
\Node (5) at (5,0){};
\Node (6) at (6,0){};
\Node (7) at (7,0){};
\Node (8) at (8,0){};
\Node (9) at (9,0){};
\Node (10) at (10,0){};
\draw (1) to[bend left] (3);
\draw (4) to[bend left] (10);
\draw (5) to[bend left] (9);
\draw (6) to[bend left] (8);
\draw (1) to[bend right] (2);
\draw (3) to[bend right] (7);
\draw (4) to[bend right] (6);
\draw (8) to[bend right] (10);
\end{tikzpicture}
\hspace{1em}
\begin{tikzpicture}[scale=.4]
\def\Node{\node [circle, fill, inner sep=2pt]}
\node at (-1,0){$\overset{R}{\longrightarrow}$};
\node at (4.5,-2){$M_{9}^\A\frac{2 \dd 7}{1 \dd 5 \dd 3}$};
\Node (1) at (1,0){};
\Node (2) at (2,0){};
\Node (3) at (3,0){};
\Node (4) at (4,0){};
\Node (5) at (5,0){};
\Node (6) at (6,0){};
\Node (7) at (7,0){};
\Node (8) at (8,0){};
\Node (9) at (9,0){};
\draw (1) to[bend left] (2);
\draw (3) to[bend left] (9);
\draw (4) to[bend left](8);
\draw (5) to[bend left](7);
\draw (2) to[bend right](6);
\draw (3) to[bend right](5);
\draw (7) to[bend right](9);
\end{tikzpicture}
\hspace{1em}
\begin{tikzpicture}[scale=.4]
\def\Node{\node [circle, fill, inner sep=2pt]}
\node at (-1,0){$\overset{B}{\longrightarrow}$};
\node at (6,-2){$M_{8}^\A\frac{1 \dd 7}{5 \dd 3}$};
\Node (1) at (1,0){};
\Node (2) at (2,0){};
\Node (3) at (3,0){};
\Node (4) at (4,0){};
\Node (5) at (5,0){};
\Node (6) at (6,0){};
\Node (7) at (7,0){};
\Node (8) at (8,0){};
\draw (2) to[bend left] (8);
\draw (3) to[bend left] (7);
\draw (4) to[bend left] (6);
\draw (1) to[bend right](5);
\draw (2) to[bend right](4);
\draw (6) to[bend right](8);
\end{tikzpicture}
\hspace{1em}
\begin{tikzpicture}[scale=.4]
\def\Node{\node [circle, fill, inner sep=2pt]}
\node at (-1,0){$\overset{F}{\longrightarrow}$};
\node at (4,-2){$M_{8}^\A\frac{5 \dd 3}{1 \dd 7}$};
\Node (1) at (1,0){};
\Node (2) at (2,0){};
\Node (3) at (3,0){};
\Node (4) at (4,0){};
\Node (5) at (5,0){};
\Node (6) at (6,0){};
\Node (7) at (7,0){};
\Node (8) at (8,0){};
\draw (1) to[bend left] (5);
\draw (2) to[bend left] (4);
\draw (6) to[bend left] (8);
\draw (2) to[bend right](8);
\draw (3) to[bend right](7);
\draw (4) to[bend right](6);
\end{tikzpicture}$$
$$\begin{tikzpicture}[scale=.4]
\def\Node{\node [circle, fill, inner sep=2pt]}
\node at (-1,0){$\overset{P}{\longrightarrow}$};
\node at (3.5,-2){$M_{7}^\A\frac{3 \dd 1 \dd 3}{7}$};
\Node (1) at (1,0){};
\Node (2) at (2,0){};
\Node (3) at (3,0){};
\Node (4) at (4,0){};
\Node (5) at (5,0){};
\Node (6) at (6,0){};
\Node (7) at (7,0){};
\draw (1) to[bend left](3);
\draw (5) to[bend left](7);
\draw (1) to[bend right](7);
\draw (2) to[bend right](6);
\draw (3) to[bend right](5);
\end{tikzpicture}
\hspace{1em}
\begin{tikzpicture}[scale=.4]
\def\Node{\node [circle, fill, inner sep=2pt]}
\node at (-1,0){$\overset{F}{\longrightarrow}$};
\node at (3.5,-2){$M_{7}^\A\frac{7}{3 \dd 1 \dd 3}$};
\Node (1) at (1,0){};
\Node (2) at (2,0){};
\Node (3) at (3,0){};
\Node (4) at (4,0){};
\Node (5) at (5,0){};
\Node (6) at (6,0){};
\Node (7) at (7,0){};
\draw (1) to[bend left](7);
\draw (2) to[bend left](6);
\draw (3) to[bend left](5);
\draw (1) to[bend right](3);
\draw (5) to[bend right](7);
\end{tikzpicture}
\hspace{1em}
\begin{tikzpicture}[scale=.4]
\def\Node{\node [circle, fill, inner sep=2pt]}
\node at (-1,0){$\overset{P}{\longrightarrow}$};
\node at (3.5,-2){$M_{4}^\A\frac{1 \dd 3}{1 \dd 3}$};
\Node (1) at (1,0){};
\Node (2) at (2,0){};
\Node (3) at (3,0){};
\Node (4) at (4,0){};
\draw (2) to[bend left](4);
\draw (2) to[bend right](4);
\end{tikzpicture}
\hspace{1em}
\begin{tikzpicture}[scale=.4]
\def\Node{\node [circle, fill, inner sep=2pt]}
\node at (-1,0){$\overset{C(1)}{\longrightarrow}$};
\node at (3.5,-2){$M_{3}^\A\frac{3}{3}$};
\Node (1) at (1,0){};
\Node (2) at (2,0){};
\Node (3) at (3,0){};
\draw (1) to[bend left](3);
\draw (1) to[bend right](3);
\end{tikzpicture}
\hspace{1em}
\begin{tikzpicture}[scale=.4]
\node at (-1,0){$\overset{C(3)}{\longrightarrow}$};
\node at (1,-2){$\emptyset$};
\node at (1,0){$\emptyset$};
\end{tikzpicture}$$
\caption{
The meander $M_{10}^\A\frac{3 \dd 7}{2 \dd 5 \dd 3}$ has signature $RBFPFPC(1)C(3)$.}
\label{WindingDownSignature}
\end{figure}

\end{example}

The component elimination moves in Lemma \ref{WindingDown} give the homotopy type of the meander.  A meander has \textit{homotopy type} $H(a_1,a_2,\dots ,a_m)$ if its signature contains $C(a_i)$ exactly once for all integers $i \in [1,m]$ in addition to no other component elimination moves.  Further, we say if a meander has homotopy type $H(a_1,a_2,\dots ,a_m)$, then it is \textit{homotopically equivalent} to the meander $M_{\sum a_i}^\A \dfrac{a_1|a_2|\dots |a_m}{a_1|a_2|\dots |a_m}$.  
We define the \textit{homotopy type of a seaweed} to be the homotopy type of its corresponding meander.

\begin{example} The seaweed $\mf{p}_{10}^\A\frac{3 \dd 7}{2 \dd 5 \dd 3}$ has homotopy type $H(1,3)$.
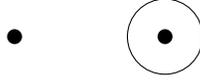
\begin{figure}[H]
\[\begin{tikzpicture}
\def\Node{\node [circle, fill, inner sep=2pt]}
\Node at (1,0){};
\Node at (3,0){};
\draw (3,0) circle (.5cm);
\end{tikzpicture}\]
\caption{The meander for $\mf{p}_{10}^\A\frac{3 \dd 7}{2 \dd 5 \dd 3}$ is homotopically equivalent to the meander $M_4^\A\frac{1|3}{1|3}$.}
\label{HomotopyType}
\end{figure}
\end{example}

Note that each of the moves in Lemma \ref{WindingDown} can be reversed to yield a ``winding-up" move.  These moves, which we record in the following lemma, can be used to build any meander of any size and block configuration.

\begin{lemma}[Coll et al. \textbf{\cite{Coll3}}, Lemma 4.2]\label{WindingUp}
Every meander is the result of a 
sequence of the following moves applied to the empty meander. 
Given the meander 
$M=M_n^\A\dfrac{a_{1}|a_{2}| \dots |a_{m}}{b_{1}|b_{2}| \dots |b_{t}}$, 
create a meander $M'$ by one of the following moves:

\begin{enumerate}[\textup(i\textup)]
\item {\bf Flip $(\widetilde{F})$:} $M'= \displaystyle\frac{b_1|b_2|...|b_t}{a_1|a_2|...|a_m}$,

\item {\bf Component Creation $(\widetilde{C}(c))$:} 
$M'=\displaystyle\frac{c|a_1|a_2|...|a_m}{c|b_1|b_2|...|b_t},$

\item {\bf Rotation Expansion $(\widetilde{R})$:}
if $a_1 > b_1$, then $M'= \displaystyle
\frac{(2a_{1} - b_{1})|a_{2}| \dots |a_{m}}{a_{1}|b_{2}| \dots |b_{t}}$,

\item {\bf Block Creation $(\widetilde{B})$:}
$M' = \displaystyle
\frac{2a_{1}|a_{2}| \dots |a_{m}}{a_{1}|b_{1}|b_{2}| \dots |b_{t}},$

\item {\bf Pure Expansion $(\widetilde{P})$:} 
$M' = \displaystyle
\frac{a_{1} + 2a_{2}|a_{3}|a_{4}|\dots|a_{m}}{a_{2}|b_{1}|b_{2}| \dots |b_{t}}.$

\end{enumerate}
\end{lemma}

\begin{remark}
All moves in Lemmas \ref{WindingDown} and \ref{WindingUp} preserve homotopy type except for the component elimination and the component creation moves. 
\end{remark}

\newpage
\bibliographystyle{abbrv}

\bibliography{bibliography.bib}

\end{document}